\documentclass[11pt]{amsart}
\usepackage{amssymb,color}
\usepackage{graphics}

\usepackage{tikz, pgfplots}

\newlength\figureheight
\newlength\figurewidth
\newsavebox{\activebox}

\usepackage{bm}

 \numberwithin{equation}{section}

\newtheorem{theorem}{Theorem}[section]
\newtheorem{lemma}[theorem]{Lemma}
\newtheorem{proposition}[theorem]{Proposition}

\newtheorem{remark}[theorem]{Remark}
\newtheorem{example}[theorem]{Example}

\newtheorem{corollary}[theorem]{Corollary}

\def\th{\text{TH}}
\def\bh{\text{BH}}
\def\Rr{{\mathbb R}}

\def\k{{\kappa}}

\def\H1{{\mathbb H}^1}

\def\R2{{\mathbb R}^2}

\def\R3d{{\mathbb R}^{3d}}
\def\RR{\mathbb R}

\def\Rmd{{\mathbb R}^{md}}
\def\Rd{{\mathbb R}^d}

\def\S{{\mathcal S}}

\def\F{{\mathcal F}}

\def\TF{{\mathcal F}}
 \def\LW{{L_w^{(p_0,{\bf p}),\kappa;({\bf q},q_{m+1}),\rho}}}
 \def\MW{{\mathcal M_w^{(p_0,{\bf p}),\kappa;({\bf q},q_{m+1}),\rho}}}
 
 \def\MV{{\mathcal M_v^{(r_0,{\bf r}),\kappa;({\bf s},s_{m+1}),\rho}}}
\newcommand{\ip}[2]{\langle#1,#2\rangle}

\begin{document}

\title{Boundedness of multilinear pseudo-differential operators on modulation spaces}

\author{Shahla Molahajloo, Kasso A.~Okoudjou, and G\"otz E. Pfander}

\address{Shahla Molahajloo\\
Department of Mathematics\\  Institute for Advanced Studies in Basic Sciences (IASBS)\\
P. O. Box 45195-1159\\ Gava Zang, Zanjan 45137-66731 Iran}
 \email{Molahajloo@iasbs.ac.ir}

\address{Kasso A.~Okoudjou\\
Department of Mathematics\\
 University of Maryland\\
 College Park, MD, 20742 USA}
\email{kasso@math.umd.edu}

\address{G\"otz E. Pfander\\
School of Science and Engineering\\
 Jacobs  University\\ 28759 Bremen, Germany}
\email{g.pfander@jacobs-university.de}


\date{\today}

\maketitle \pagestyle{myheadings} \thispagestyle{plain}
\markboth{S. MOLAHAJLOO, K. A. OKOUDJOU, G. E. PFANDER}{SYMBOL CLASSES FOR MULTILINEAR PSEUDO-DIFFERENTIAL OPERTORS}

\noindent
\begin{abstract}
Boundedness results for multilinear pseudodifferential operators on products of modulation spaces  are derived based on ordered integrability conditions on the short-time Fourier transform of the operators' symbols. The flexibility and strength of the introduced methods is demonstrated by their application to  the bilinear and trilinear Hilbert transform.  
\end{abstract}

\section{Introduction and motivation}\label{sec1}

Pseudodifferential operators  have long been studied in the context of partial differential equations \cite{Hor2, Hor1,Kumano, stein93, taylor,WongPseudo, WongWeyl}. Among the most investigated topics on such  operators are  minimal smoothness and decay conditions  on their symbols   that guarantee their boundedness on function spaces of interest.  In recent years, results from time-frequency analysis have been exploited to obtain boundedness results on so-called modulation spaces, which in turn yield boundedness on Bessel potential spaces, Sobolev spaces, and Lebesgue spaces via well established embedding results. 
In this paper, we develop  time-frequency analysis based methods in order to establish boundedness of classes \emph{multilinear} pseudodifferential operators on products of modulation spaces.
%

\subsection{Pseudodifferential operators}\label{sub1.1}
A pseudodiffrential operator is  an operator $T_\sigma$ formally defined through its symbol $\sigma$ by $$T_\sigma f(x)=\int_{\RR^d}\sigma(x, \xi) \hat{f}(\xi) \,e^{2\pi i x\cdot \xi} d\xi,$$ where the Fourier transformation is formally given by
$\left(\F f \right)(\xi)=\widehat f(\xi)=\int_{\Rd}e^{-2\pi ix\cdot\xi}f(x)\,dx.$
H\"ormander symbol classes are arguably the most used in investigating pseudodifferential operators.
  In particular, the class of smooth symbols with bounded derivatives was shown to yield bounded operator on $L^2$ in the celebrated work of Calder\'on and Vaillancourt \cite{cava}. More specifically, if $\sigma \in S_{0,0}^0$, that is, for all non-negative integers $\alpha,\beta$ there exists $C_{\alpha,\beta}$ with
\begin{equation}\label{linearcz-condition}|\partial_{x}^\alpha \partial_\xi^\beta\sigma(x, \xi)|\leq C_{\alpha, \beta},\end{equation}  then $T_\sigma$ maps $L^2$ into itself. 

\subsection{Time-frequency analysis of pseudodifferential operators}\label{sub1.2} In \cite{Sjostrand1}, J.~Sj\"ostrand   defined a class of bounded operators on $L^2$ whose symbols do not have to satisfy a  differentiability assumption and which contains those operators with symbol in $S_{0,0}^0$. He  proved that this class of symbols forms an algebra under the so-called twisted convolution \cite{fol89, Groch2, Sjostrand1, Sjostrand2}.  Incidentally, symbols of Sj\"ostrand's  class operators are characterized by their membership in the \emph{modulation space} $M^{\infty, 1}$, a space of tempered distributions introduced  by Feitchinger via integrability and decay conditions on the distributions' \emph{short-time Fourier transform} \cite{fei83}.  
 Gr\"ochenig and Heil  then significantly extended  Sj\"ostrands results  by establishing the boundedness of his pseudodifferential operators on all modulation spaces \cite{GrochHeil1}. 
 
These and similar results on pseudodifferential operators were recently extended by Molahajloo and Pfander through the introduction of ordered integrability conditions on the short-time Fourier transform of the operators' symbols \cite{mopf}.  Similar approaches have  been used to derive other boundedness results of pseudodifferential operators on modulation space like spaces \cite{Shannon}. The approach of varying integration orders of short-time Fourier transforms of, here, symbols of multilinear operators lies  at the center of this paper. 
 
Today, the  functional analytical tools developed to analyze pseudodifferential operators on modulation spaces form an integral part of time-frequency analysis.   They are used, for example, to model time-varying filters prevalent in signal processing. By now, a robust body of work stemming from this point of view has been developed \cite{Czaja, GrochHeil1, GrochHeil2, hrt1, rochtach, Toft1, Toft2, Toft3, tach}, and has lead to a number of applications to areas such as seismic imaging, and communication theory \cite{magigro, Stro}.

\subsection{Multilinear pseudodifferential operators}\label{sub1.3}  A multilinear pseudo-differential operator $T_\sigma$ with distributional symbol $\sigma$  on $\RR^{(m+1)d}$, is formally given by  
\begin{equation}\label{mpsido}
 \left(T_{\sigma} {\bm f}\right)(x)=\int_{\RR^{md}}e^{2\pi i x\cdot (\sum_{i=1}^d \xi_i)}\sigma(x,\bm{\xi})\widehat{ {f_1}}(\xi_1)  \widehat{ {f_2}}(\xi_2)  \ldots   \widehat{ {f_m}}(\xi_m)\,d\bm{\xi}.
\end{equation}
Here and in the following we use boldface characters as $\bm \xi=(\xi_1,\ldots,\xi_m)$ to denote products of $m$ vectors  $\xi_i\in \RR^d$, and it will not cause confusion to use the symbol $\bm f$ for both,  a vector of $m$ functions or distributions $\bm f = (f_1,\ldots,f_m)$, that is, a vector valued function or distribution on $\RR^d$, and the rank one tensor  $\bm f = f_1\otimes \ldots \otimes f_m$, a function or distribution on $\RR^{md}$. For example, we write $\widehat{\bm f}(\bm \xi)=\widehat{ f_1}( \xi_1)\cdot \ldots \cdot \widehat{ f_m}( \xi_m)$, while $\widehat{\bm f}( \xi)=(\widehat{ f_1}( \xi), \ldots , \widehat{ f_m}( \xi))$.

A trivial example of a multilinear operator is given by the constant symbol $\sigma \equiv 1$. Clearly, $T_\sigma(\bm f)$ is simply the product  $f_1(x)f_2(x)\ldots f_m(x)$. Thus, H\"older's inequality determines boundedness on products of Lebesgue spaces. On the other hand, when the symbol is independent of the space variable $x$, that is, when $\sigma(x, \bm{\xi})\equiv \tau(\bm{\xi})$, the $T_\sigma=T_\tau$ is a multilinear Fourier multipliers.  We refer to \cite{beto, betzi, come, gratore, come2, mtt} and the references therein for a small sample of the  vast literature on multilinear pseudodiffrential operators.

One of the questions that has been repeatedly investigated relates to  (minimal) conditions on the symbols $\sigma$ that would guarantee the boundedness of ~\eqref{mpsido} on  products of certain function spaces, see \cite[Theorem 34]{come}. For example, one can ask if a multilinear version of~\eqref{linearcz-condition} exist. B\'enyi  and Torres  (\cite{beto}) proved that  unless additional conditions are added, there exist symbols which satisfy such multilinear estimates but for which the corresponding multilinear pseudodifferential  operators are unbounded on products of certain Lebesgue spaces. Indeed, in the bilinear case, that is, when $m=2$,    the  class of  operators  whose 
symbols satisfy for all non-negative integers $\alpha, \beta, \gamma$, 
\begin{equation} \label{bicv}
|\partial_x^\alpha \partial_\xi^\beta \partial_\eta^\gamma \sigma(x,\xi,\eta)|
\leq C_{\alpha,\beta,\gamma}
\end{equation} 
contains  operators that do not map   $L^2 {\times} L^2$
into $L^1$. 

Multilinear pseudodifferential operators in the context of their boundedness on modulation spaces, were first investigated in \cite{BeOk04, BGHO}. Results obtained in this setting have been used to establish  well posedness  for a number of non-linear PDEs in these spaces \cite{bgor07, abekok07}. For example, and as opposed to the classical analysis of multilinear pseudodifferential operators, it was proved in~\cite{BGHO} that symbols satisfying~\eqref{bicv} yield boundedness from $L^2\times L^2$ into the modulation space $M^{1, \infty}$, a space  that contains $L^1$. 
The  current paper offers some new insights and results in this line of investigation.

\subsection{Our contributions}\label{sub1.4} Modulation spaces are defined by imposing integrability conditions on the short-time Fourier transform of the distribution at hand. 
Following ideas from Molahajloo and Pfander \cite{mopf}, we  impose various ordered integrability conditions on the short-time Fourier transform of  a tempered distribution $\sigma$ on $\RR^{(m+1)d}$ which is a symbol of a multilinear pseudodifferential operator.  By using this new setting, we establish  new boundedness results for multilinear pseudodifferential operators on products of modulations spaces. For example, the following result follows from  our main result, Theorem~\ref{thm:bilinear}.

\begin{theorem}\label{model-result} If $1\leq p_0, p_1, p_2, q_1, q_2, q_3 \leq \infty$ satisfy $$ \tfrac{1}{p_0}\leq \tfrac{1}{p_1}+ \tfrac{1}{p_2}\quad {\textrm and}\quad 1+ \tfrac{1}{q_3}\leq \tfrac{1}{q_1}+ \tfrac{1}{q_2},$$ and if  for some Schwartz class function $\varphi$, the symbol short-time Fourier transform 
$$\mathcal V_{\bf \varphi} \sigma(x,t_1,t_2,\xi_1,\xi_2,\nu)= \int \!\!\! \int \!\! \! \int \sigma(\widetilde x,\widetilde{\xi_1},\widetilde{\xi_2}) \varphi(x-\widetilde x)\varphi(\xi_1 -\widetilde {\xi_1})\varphi(\xi_2 -\widetilde {\xi_2})) \,e^{-2\pi i (x\nu-t_1\xi_1-t_2\xi_2)}d\widetilde x \,d\widetilde {\xi_1} \,d\widetilde {\xi_2}$$
satisfies
\begin{align}\label{en:defanmod}
  \|\sigma\|_{\mathcal{M}^{(\infty, 1,1); (\infty, \infty, 1)}}=\int \sup_{\xi_1,\xi_2}\iint \sup_{x} |\mathcal V_{\bf \varphi} \sigma(x,t_1,t_2,\xi_1,\xi_2,\nu)|\, dt_1 \,dt_2 \,d\nu < \infty,
\end{align}
then the pseudodifferential operator $T_{\sigma}$ initially defined on $S(\Rd) \times
S(\Rd)$ by
\begin{equation*}
T_{\sigma} (f_1,f_2)(x)=\iint e^{2\pi i x\cdot ( \xi_1+\xi_2)}\sigma(x,\xi_1,\xi_2)\widehat{ {f_1}}(\xi_1)  \widehat{ {f_2}}(\xi_2) \, d{\xi_2}  \, d{\xi_1}
\end{equation*}
extends to a bounded bilinear operator from
$M^{p_1,q_1} \times M^{p_2,q_2}$ into $M^{p_0,q_3}$.
Moreover, there exists a constant $C>0$ that only depends on $d,$ the $p_i$, and $q_i$ with
$$
 \|T_{\sigma}(f_1,f_2)\|_{M^{p_0,q_3}}\leq C \|\sigma\|_{\mathcal{M}^{(\infty, 1,1); (\infty, \infty, 1)}}\ 
\|f_1\|_{M^{p_1,q_1}}\ \|f_2\|_{M^{p_2,q_2}}.$$ 
\end{theorem}

We note that the classical modulation space $M^{\infty,1}(\RR^{3d})$ can be continuously embedded into   $ \mathcal{M}^{(\infty, 1,1), (\infty, \infty, 1)}(\RR^{3d})$ implicitly defined by \eqref{en:defanmod}. Indeed,   
\begin{align*}
  \|\sigma\|_{\mathcal{M}^{(\infty, 1,1); (\infty, \infty, 1)}}
  	&=\int \sup_{\xi_1,\xi_2}\iint \sup_{x} |\mathcal V_{\bf \varphi} \sigma(x,t_1,t_2,\xi_1,\xi_2,\nu)|\, dt_1 \,dt_2 \,d\nu  \\
	&\leq \int \!\!\!  \int \!\!\!  \int  \sup_{x,\xi_1,\xi_2} |\mathcal V_{\bf \varphi} \sigma(x,t_1,t_2,\xi_1,\xi_2,\nu)|\, dt_1 \,dt_2 \,d\nu =\|\sigma\|_{{M}^{\infty; 1}}.
\end{align*}
As a consequence
Theorem~\ref{model-result} already extends the main result, Theorem 3.1, in \cite{BGHO}. 

The herein presented new approach allows us to investigate the boundedness of the bilinear Hilbert transform on products of  modulation spaces. Indeed,   in the one dimensional setting, $d=1$, it can be shown that the symbol of the bilinear Hilbert transform $$\sigma_H \in \mathcal{M}^{(\infty, 1,r); (\infty, \infty, 1)}\setminus \mathcal{M}^{(\infty, 1,1); (\infty, \infty, 1)}$$
 for all $r>1$. Hence, $\sigma_H \not\in M^{\infty, 1}$ and existing methods to investigate multilinear pseudodifferential operators on products of modulations spaces are not applicable. Using the techniques developed below, we obtain novel and wide reaching boundedness results for the bilinear Hilbert transform on the product of modulation spaces.  For example, as a special case of our result, we prove that the bilinear Hilbert transform is bounded from $L^2 \times L^2$ into the modulation space $M^{1+\epsilon, 1}$ for any $\epsilon>0$.
 
The results  established here aim at generality and differ in technique from the ground breaking results about the bilinear Hilbert transformed  as obtained by Lacey and Thiele \cite{mlct97, ml98, mlct99, mlct00}. They are therefore not easily compared  to those obtained using ``hard analysis'' techniques.  
Nonetheless, using our results and some embeddings of modulation spaces into Lebesgue space, we discuss the relation of our results on the boundedness of the bilinear Hilbert transform to the known classical results. 

The herein given framework is flexible enough to allow  an initial investigation of the trilinear Hilbert transform. Here we did not try to optimize our results but just show through some examples how one can tackle this more difficult operator in the context of  modulation spaces. 
%

\subsection{Outline}\label{sub1.5}  We introduce our new class of symbols based on a modification of the short-time Fourier transform in  Section~\ref{sec2}. We then  prove a number of technical results including some Young-type inequalities, that form the foundation of   our main results.  Section~\ref{sec3} contains most of the key results needed to establish our results. This naturally leads to our main results concerning the boundedness of multilinear pseudodifferential operators on product of modulation spaces. 
Section~\ref{sec:applications} is devoted to applications of our results. In Section~\ref{subsec4.1} we specialize our results to the bilinear case, proving boundedness results of bilinear pseudodifferential operators on products of modulation spaces. We then consider as example the bilinear Hilbert transform in Section~\ref{sub4.2}. In Section~\ref{sec5} we initiate an investigation of the boundedness of the   trilinear Hilbert transform on products of  modulation spaces.

%
%

 \section{Symbol classes for multilinear pseudodifferential operators}\label{sec2}
\subsection{Background on modulation spaces}\label{subsec2.1}

Let ${\bf{r}}=(r_1, r_2, \dots, r_m)$  where $1\leq r_i<\infty$, $i=1,2,\dots, m$. The mixed norm space $L^r({\Rr}^{md})$  is Banach space of measurable functions $F$ on ${\Rr}^{md}$ with finite norm  \cite{benedek}
\begin{eqnarray}
&&\hspace{-1cm}\|F\|_{L^{\bf{r}}}=\Big(\int_{\Rr^d}\dots\Big(\int_{\Rr^d}\Big(\int_{\Rr^d}|F(x_1,
\dots,x_m)|^{r_1}\,dx_1\Big)^{r_2/r_1}\,dx_2\dots\Big)^{r_m/r_{m-1}}\,dx_m\Big)^{1/r_m}. \nonumber
\end{eqnarray}
Similarly, we define $L^{\bf{r}}(\Rr^{md})$ where $r_i=\infty$ for some indices $i$.
For a nonnegative measurable function $w$  on $\Rr^{md}$ wee define  $L^{\bf{r}}_w(\Rr^{md})$ to be the space all   $F$ on $\Rr^{md}$ for
which $Fw$ is in $L^{\bf{r}}(\Rr^{md})$, that is, $
\|F\|_{L^{\bf{r}}_w}=\|Fw\|_{L^{\bf{r}}}<\infty.
$

For the purpose of this paper, we define a mixed norm space depending on a permutation that determines the  order of integration. For a permutation  $\rho$ on $\{1,2,\ldots,n\}$, the weighted mixed norm space $L_w^{\bm r;\rho}(\Rr^{md})$  is the set of all measurable functions $F$ on $\Rr^{md}$ for which
  %
   \begin{align*}
  & \|F\|_
{L_w^{\bm r;\rho}} =
\Big(\int_{\Rd}\Big(\int_{\Rd}\Big(\ldots \Big(\int_{\Rd} 
  \qquad   |  F
(x_1,x_2,\ldots, x_n)\, w(x_1,x_2,\ldots, x_n)|^{r_{\rho(1)}} \\ 
& \qquad \qquad \qquad \qquad \,dx_{\rho(1)}\Big)^{r_{\rho(2)}/r_{\rho(1)}}\,dx_{\rho(2)}\Big)^{r_{\rho(3)}/r_{\rho(2)}}  \ldots dx_{\rho(n)}\Big)^{1/r_{\rho(n)}}\,\nonumber
  \end{align*}
  is finite.

Let $M_\nu$ denote  modulation by $\nu\in\Rd$, namely,  $M_\nu f(x)=e^{2\pi i t\cdot \nu}f(x)$, and let $T_t$ be translation by $t\in\Rd$, that is,  $T_t f (x)=f(x-t)$.
The short-time Fourier transform $V_\phi f$ of $f\in \mathcal S'(\Rd)$ with respect to the Gaussian window $\phi(x)=e^{-\|x\|^2}$ is given by
  \begin{eqnarray*}
    V_\phi f(t,\nu)=\F\big( f\,T_{t}\phi\big)(\nu)=(f,M_\nu T_t\phi)=\int f(x)\ e^{-2\pi i x \nu}\phi(x-t)\, dx\,.
  \end{eqnarray*}
The modulation space $M^{p,q}(\Rd)$, $1\leq p,q \leq \infty$,  is a Banach space consisting of those $f\in \mathcal S'(\Rd)$ with
  \begin{eqnarray*}
   \|f\|_{M^{p, q}} =\|V_\phi f \|_{L^{p, q}}= \Big( \int\Big( \int |V_\phi f(t,\nu)|^p\, dt\Big)^{q/p}\, d\nu   \Big)^{1/q} < \infty\,,
  \end{eqnarray*}
with usual adjustment of the mixed norm space if $p=\infty$ and/or $q=\infty$. We refer to \cite{fei83, Groch2} for  background on  modulation spaces.


In the sequel we consider weight functions $w$ on $\RR^{2(m+1)d}$. We
assume that $w$ is continuous and  sub-multiplicative, that is, 
$w(x+y)\leq C w(x)w(y).$
Associated to $w$ will be a family of $w$-moderate weight functions
$v$. That is $v$ is positive, continuous and satisfies $v(x+y)\leq C
w(x)v(y).$ 

\subsection{A new class of symbols}\label{subsec2.2}
The commonly used short-time Fourier transform analyzes functions in time\footnote{For clarity,  we always  refer to the variables $x$,$y$,$\bm t$ as time variables, even though a physical interpretation of time necessitates $d=1$. Alternatively, one can consider multivariate $x$,$y$,$\bm t$ as spatial variables.}; as symbols have time and frequency variables, we base the herein used short-time Fourier transform on a Fourier transform that takes Fourier transforms in time variables and inverse Fourier transforms in frequency variables. We  then order the variables, first time, then frequency. That is, we follow the idea of symplectic Fourier transforms $\mathcal F_s$ on phase space,
$$ \mathcal F_sF (\bm t,\nu)=\iint_{\RR^{(m+1)d}}F(x,\bm \xi)\, e^{2\pi i(\bm \xi\bm t- x\nu)}
d\bm \xi dx .$$
For $F\in \mathcal S'(\RR^{(m+1)d})$  and $\phi \in \mathcal S(\RR^{(m+1)d})$, we define the  {\em symbol short-time Fourier transform} $\mathcal V_\phi F $ of $F$ with respect to $\phi$  by
  \begin{align}
    \mathcal V_\phi F(x,\bm t , \bm \xi ,\nu)
    		&= \TF_s\big( F\,T_{(x,\bm \xi )}\phi\big)( \bm t,\nu)= \langle F,M_{(-\nu, \bm t)} T_{(x,\bm \xi)}\phi\rangle \notag \\
    		&=\int_{\RR^{md}}\int_{\Rd}e^{-2\pi i( \widetilde{x}\nu-\bm t \widetilde{\bm \xi})}
			F(\widetilde{ x},\widetilde{\bm \xi}, )\phi(\widetilde{x}-x,\widetilde{\bm \xi}-\bm \xi)\,
    				d\widetilde{x}\,d\widetilde{\bm \xi}\notag 
  \end{align} where $x, \nu \in \Rd, $ and $\bm t, \bm \xi \in \RR^{md}$. 
Note that the symbol short-time Fourier transform is related to the ordinary short-time Fourier transform by
$$\mathcal V_\phi F(x,\bm t,\bm \xi,\nu)= V_\phi F(x,\bm \xi,\nu,-\bm t).$$ 

 Modulation spaces for  symbols of multilinear operators are then defined by requiring the symbol short-time Fourier transform of an operator to be in certain weighted $L^p$ spaces.  To describe these, we fix decay parameters $1\leq p_0,p_1,\ldots,p_m,q_1,q_2,\ldots, q_m, q_{m+1}\leq \infty$, and permutations  $\kappa$    on $\{0,1,\ldots,m\}$ and $\rho$  on $\{1,\ldots,m,m+1\}$. The latter indicate the integration order   of the time, respectively frequency, variables. Put, 
${\bf p}=(p_1,p_2,\dots,p_m)$, ${\bf q}=(q_1,q_2,\dots, q_m)$ and  let $w$ be a weight function on $\RR^{2(m+1)d}$.
 Then 
$\LW(\RR^{2(m+1)d})$
 is the mixed norm space consisting of those measurable functions $F$ for which the norm 
   \begin{align*}
  & \|F\|_
\LW \\
  &=\Big(\int_{\Rd}\Big(\int_{\Rd}\Big(\ldots \Big(\int_{\Rd}\Big(\int_{\Rd}\Big(\ldots \Big(\int_{\Rd}\Big(\int_{\Rd} \\
  &\qquad   |w(t_0,t_1,\ldots, t_m,\xi_1,\ldots, \xi_m,\xi_{m+1})\  F
(t_0,t_1,\ldots, t_m,\xi_1,\ldots, \xi_m,\xi_{m+1})|^{p_{\kappa(0)}} \\ & \qquad \quad dt_{\kappa(0)}\Big)^{p_{\kappa(1)}/p_{\kappa(0)}}  \,dt_{\kappa(1)}\Big)^{p_{\kappa(2)}/p_{\kappa(1)}}  \ldots  dt_{\kappa(m)}\Big)^{q_{\rho(1)}/p_{\kappa(m)}}\,d\xi_{\rho(1)}\Big)^{q_{\rho(2)}/q_{\rho(1)}}  
\ldots \,d\xi_{\rho(m+1)} \Big)^{1/q_{\rho(m+1)}}\nonumber
  \end{align*}
  is finite.
The weighted {\em symbol modulation space}  
$\MW(\RR^{(m+1)d})$ is composed of those  $F\in \mathcal S'(\RR^{(m+1)d})$ with 
   \begin{align*}
  & \|F\|_
{\MW} =\|\mathcal V_\phi F \|_
{\LW}
<\infty\,.\nonumber
  \end{align*}

When  $\kappa$ and $\rho$ are  identity permutations, then we denote $\LW(\RR^{2(m+1)d})$ and $\MW(\RR^{2(m+1)d})$ by $ L_w^{(p_0,{\bf p});({\bf q},q_0)}(\RR^{2(m+1)d})$ and ${\mathcal M_w^{(p_0,{\bf p});({\bf q},q_0)}}(\RR^{2(m+1)d})$,  respectively.
The dependence of the norm on the choice of $\kappa, \rho$, as well as the advantage of choosing a particular order will be discussed in Section~\ref{sec:young permute}.

For simplicity of notation, we  set 
 $S(\bm \xi)=\sum_{i=1}^{m}\xi_{i}$. 
For functions $g$ and components of $\bm f$   in $\S(\Rd)$, the Rihaczek transform $R(\bm f,g)$
of $\bm f$ and $g$ is defined by
$$R\left(\bm f,g\right)(x,\bm\xi)=e^{2\pi i x\cdot(\xi_1+\ldots+ \xi_m)}\widehat f_1(\xi_1)\cdot \ldots \cdot \widehat f_m(\xi_m)\overline{g(x)}
=e^{2\pi i x\cdot S(\bm \xi)}\widehat{ \bm f }( \bm \xi)
\overline{g(x)}.$$
Multilinear pseudo-differential operators are related to  Rihaczek transforms by
$$\ip{T_\sigma \bm f}{ g}=\ip{\sigma}{\overline{R(\bm f,g)}}$$
a-priori for all
functions $f_i$ and $g$ in $\S(\Rd)$ and symbols $\sigma\in\S(\RR^{(m+1)d})$. 

With $x\pm\bm t=x\pm(t_1,\ldots, t_m)= (x\pm t_1,\ldots,x\pm t_m)$,
it can be easily seen that
$$R\left(\bm f,g\right)(x,\bm\xi)=\F_{\bm t\to\bm \xi}\left(\bm f  (\cdot + x)\right)\overline{g}(x)$$
where
$$\F_{\bm t\to\bm \xi}\left(\bm f(\cdot+ x)\right)(\bm \xi)=\int_{\RR^{md}}e^{-2\pi i \bm t  \cdot \bm\xi}\bm f  (\bm t+ x)\,d\bm t.$$

%

\begin{lemma}\label{lemmaT_A}
For $\varphi$  real-valued,
$\bm\varphi=(\varphi,\ldots,\varphi), \bm f =(f_1, f_2,\ldots,f_m) \in\S(\Rd)^m$, and $g\in\S(\Rd)$,
$$
	 V_{T_A(\bm\varphi\otimes\varphi)}T_A(\overline{\bm f}\otimes g)(x,-\bm \xi,\bm t, \nu)=
\overline{  V_{\varphi}  f_1 (x-t_1, \xi_1) \dots   V_{\varphi}  f_m (x-t_m,\xi_m) } 
\cdot V_{\varphi}g (x,\nu-S(\bm \xi)).
 $$
 Moreover,
 \begin{eqnarray}
&&\left(V_{\overline{R(\bm \varphi,\varphi)}}\overline{R(\bm f,g)}\right)(x, \bm \xi, \nu, \bm t)
=e^{-2\pi i\bm\xi  \bm t}\left(\mathcal V_{T_{A}(\bm\varphi\otimes\varphi)}T_{A}(\overline{ \bm f}\otimes g)\right)(x,-\bm t,\nu, \bm\xi),\nonumber
\end{eqnarray}
and
in particular, $$|\left(V_{\overline{R(\bm \varphi,\varphi)}}\overline{R(\bm f,g)}\right)(x, \bm \xi, \nu, \bm t)
|=|{V}_{T_{A}(\bm\varphi\otimes\varphi)}T_{A}(\overline{ \bm f}\otimes g)(x, -\bm \xi, -\bm t, \nu)|$$
\end{lemma}
\begin{proof}
We compute
\begin{eqnarray}
&&\left( V_{T_A(\bm \varphi\otimes\varphi)}T_A(\overline{\bm f}\otimes g)\right)(x,-\bm \xi,\bm t,\nu)\nonumber\\
&=&\int_{\RR^{md}}\int_{\Rd}e^{-2\pi i(\widetilde{x}\nu+\widetilde{\bm t}{\bm \xi})}
T_A\left(\overline{ \bm f}\otimes g\right)(\widetilde{x},\widetilde{\bm t})
T_A\left(\bm \varphi\otimes\varphi\right)(\widetilde{x}-x,\widetilde{\bm t}-\bm t)
\,d\widetilde{x}\,d\widetilde{\bm t}\nonumber\\
&=&\int_{\RR^{d}}\left(\int_{\RR^{md}}e^{-2\pi i\widetilde{\bm t}\bm \xi}\overline{\bm f}(\widetilde{x}-\widetilde{\bm t})\bm \varphi(\widetilde{x}-x-\widetilde{\bm t}+\bm t)
\,d\widetilde{\bm t}\right)e^{-2\pi i\widetilde{x}\nu}g(\widetilde{x})\varphi(\widetilde{x}-x)\,d\widetilde{x}\nonumber\\
&=&\int_{\Rd}\int_{\RR^{md}} \overline{ \bm f}(\bm s) g(\widetilde{x})e^{-2\pi i(\nu \widetilde{x}+\bm \xi(\widetilde{x}-\bm s))}
\bm \varphi(\bm s-(x-\bm t))\varphi(\widetilde{x}-x)\,d\widetilde{x}\,d\bm s\nonumber\\
&=&
\left\{\overline{\int_{\RR^{md}}e^{-2\pi i \bm \xi \bm s}\bm f(\bm s)\bm \varphi(\bm s-(x-\bm t))\,d\bm s}\right\}
\left\{\int_{\Rd}e^{-2\pi i(\nu+S(\bm \xi))\widetilde{x}}g(\widetilde{x})\varphi(\widetilde{x}-x)\,d\widetilde{x}\right\}
\nonumber\\
&=&\overline{\left( V_{\bm \varphi} \bm f\right)(x-\bm t,\bm \xi)}\left(V_{\varphi}g\right)(x,\nu+S(\bm \xi)).\nonumber
\end{eqnarray}

Further,
\begin{eqnarray}
&&\left(V_{\overline{R(\bm\varphi,\varphi)}}\overline{R(\bm f,g)}\right)(x,\bm\xi,\nu,\bm t)\nonumber\\
&=&\int_{\Rmd}\int_{\Rd}e^{-2\pi i(\nu\widetilde{x}+\bm t\widetilde{\bm\xi})}\overline{R(\bm f,g)(\widetilde{x},\widetilde{\bm\xi})}
{R(\bm\varphi,\varphi)}(\widetilde{x}-x,\widetilde{\bm\xi}-\bm\xi)\,d\widetilde{x}\,d\widetilde{\bm\xi}\nonumber\\
&=&\int_{\Rmd}\int_{\Rd}e^{-2\pi i(\nu\widetilde{x}+\bm t\widetilde{\bm \xi})}
\F_{\widetilde{\bm t}\to\widetilde{\bm\xi}}\left(\overline{ \bm f}(\widetilde{x}-\cdot)\right)g(\widetilde{x})
\overline{\F_{\widetilde{\bm t}\to\widetilde{\bm\xi}-\xi}\left(\bm\varphi(\widetilde{x}-x-\cdot)\right)}\varphi(\widetilde{x}-x)\,d\widetilde{x}\,d\widetilde{\bm\xi}\nonumber\\
&=&\nonumber \int_{\Rmd}\int_{\Rd}e^{-2\pi i(\nu\widetilde{x}+\bm t\widetilde{\bm\xi})}
\F_{\widetilde{\bm t}\to\widetilde{\bm\xi}}\left(\overline{ \bm f}(\widetilde{x}-\cdot)\right)g(\widetilde{x})
{\F_{\widetilde{\bm t}\to\bm\xi-\widetilde{\bm\xi}}\left(\bm\varphi(\widetilde{x}-x-\cdot)\right)}\varphi(\widetilde{x}-x)\,d\widetilde{x}\,d\widetilde{\bm\xi}.
\end{eqnarray}
On the other hand, by using Parseval identity we have
\begin{eqnarray}
&&\left(V_{{T_A({\bm\varphi}\otimes\varphi)}}{T_A(\overline{ \bm f}\otimes g)}\right)(x,\bm t,\nu,\bm\xi)\nonumber\\
&=&\int_{\Rd}\int_{\Rmd}e^{-2\pi i(\widetilde{x}\nu+\widetilde{\bm t}{\bm\xi})}
T_A\left(\overline{\bm f}\otimes g\right)(\widetilde{x},\widetilde{\bm t})
T_A\left(\bm\varphi\otimes\varphi\right)(\widetilde{x}-x,\widetilde{\bm t}-\bm t)
\,d\widetilde{x}\,d\widetilde{\bm t}\nonumber\\
&=&\int_{\Rd}\left(\int_{\Rmd}e^{-2\pi i\widetilde{\bm t}\bm \xi}\overline{ \bm f}(\widetilde{x}-\widetilde{\bm t})\bm\varphi(\widetilde{x}-x-\widetilde{\bm t}+\bm t)
\,d\widetilde{\bm t}\right)e^{-2\pi i\widetilde{x}\nu}g(\widetilde{x})\varphi(\widetilde{x}-x)\,d\widetilde{x}\nonumber\\
&=&\int_{\Rd}\int_{\Rmd}\F_{\widetilde{\bm t}\to\widetilde{\bm\xi}}\left(\overline{\bm f}(\widetilde{x}-\cdot)\right)
\F^{-1}_{\widetilde{\bm t}\to\widetilde{\bm\xi}}\left(e^{-2\pi i\widetilde{\bm t}\bm\xi}\bm\varphi(\widetilde{x}-x+\bm t-\cdot)\right)
e^{-2\pi i\widetilde{x}\nu} g(\widetilde{x})\varphi(\widetilde{x}-x)
\,d\widetilde{\bm\xi}\,d\widetilde{x}\nonumber.
\end{eqnarray}
But,
$$\F^{-1}_{\widetilde{\bm t}\to\widetilde{\bm \xi}}\left(e^{-2\pi i\widetilde{\bm t}\bm\xi}\bm\varphi(\widetilde{x}-x+\bm t-\cdot)\right)=
e^{-2\pi i\bm t(\bm\xi-\widetilde{\bm\xi})}\F_{\bm\gamma\to\bm\xi-\widetilde{\bm\xi}}\left(\bm\varphi(\widetilde{x}-x-\cdot)\right),$$
therefore,
\begin{eqnarray}
&&\left(V_{{T_A(\bm\varphi\otimes\varphi)}}{T_A(\overline{ \bm f}\otimes g)}\right)(x,\bm t,\nu,\bm\xi)=\nonumber\\
&&e^{-2\pi i \bm t\bm\xi}
\int_{\Rmd}\int_{\Rd}
e^{2\pi i(\bm t\widetilde{\bm\xi}-v\widetilde{x})}
\F_{\widetilde{\bm t}\to\widetilde{\bm\xi}}\left(\overline{ \bm f}(\widetilde{x}-\cdot)\right)
\F_{\widetilde{\bm t}\to\bm\xi-\widetilde{\bm\xi}}\left(\bm\varphi(\widetilde{x}-x-\cdot)\right)
g(\widetilde{x})\varphi(\widetilde{x}-x)\,d\widetilde{x}\,d\widetilde{\bm\xi}.  \nonumber \qedhere
\end{eqnarray}
\end{proof}

\subsection{Young type results}\label{subsec2.3}

The following   results are consequences of Young's inequality
and will be central  in proving our main results. 
 We use the convention
that  summation over the empty set is equal to $0$.

\begin{lemma}\label{lemma:Young1}  
 Suppose that $1\leq p_k, r_k \leq
  \infty$ for $k=0, 1, \ldots, m$ and
\begin{itemize}
\item[(A1)] $  p_k \leq r_k$, $k=1,\ldots,m$; 
\item[(A2)]   $\displaystyle \sum_{\ell =1}^{k} \frac 1{p_\ell} - \frac 1{ r_\ell} \leq \frac 1 {r_0} - \frac 1{p_{k+1}}  , \quad k=0,\ldots, m-1$;
\item[(A3)] $\displaystyle \sum_{\ell =1}^m \frac 1 { p_\ell} - \frac 1 {  r_\ell} = \frac 1 {r_{0}}-\frac 1 {p_{0}}  $;
\end{itemize}
then  
$F(x,\bm t)=\bm f (x- \bm t) g(x)$ satisfies
 \begin{align*}
 \|F\|_{L^{(r_0,\bm r)}}\leq   \|g\|_{L^{p_0}} \ \|\bm f\|_{L^{\bm p}}.
 \end{align*}

 \end{lemma}

%
%

\begin{proof}
For simplicity, we use capital letters for the reciprocals of $p_k$, $r_k$, that is, $P_k=1/p_k$, $R_k=1/r_k$, $k=0,\ldots, m$.  
 Recalling that summation over the empty set is defined as $0$,  our
 assumptions (A1) -- (A3) are simply 
 \begin{itemize}
\item[(A1)] $  P_k \geq R_k$, $k=1,\ldots,m$; 
\item[(A2)]   $\displaystyle R_0 - P_{k+1} \geq \sum_{\ell =1}^{k} P_\ell -  R_\ell , \quad k=0,\ldots, m-1$;
\item[(A3)] $\displaystyle \sum_{\ell =0}^m   R_\ell=\sum_{\ell =0}^m  P_\ell$.
\end{itemize}
 Define $1/b_1=B_1=R_0+R_1-P_1,$ and for $ k=2,\ldots,m,$ 
\begin{align*}
 1/b_k=B_k&=B_{k-1}+R_k-P_k\\ &=R_0+\sum_{\ell=1}^k R_\ell-P_\ell. 
\end{align*}
 The first application of Young's inequality below requires that 
 $$
 	p_1/r_0,\ r_1/r_0,\ b_1/r_0\geq 1\quad \text{and} \quad 1/(p_1/r_0) + 1/(b_1/r_0)=1+ 1/(r_1/r_0).
	$$ 
This translates to $R_0\geq R_1, B_1, P_1$ and $P_1+B_1= R_0+ R_1$ which is equivalent to 
$$R_0\geq R_1,\ P_1,\ R_0+R_1-P_1.$$ 
But, condition (A1) of the hypothesis implies that $P_1\geq R_1$. Thus we have, 
$R_0\geq R_1,P_1$ and $P_1\geq R_1$, that is, $R_0\geq P_1\geq R_1$.
Similarly, the successive applications of Young's inequality follow by replacing $p_1,r_1,b_1,r_0$ by $p_k,r_k,b_k,b_{k-1}$, respectively.  That is, we require
$$B_{k-1}\geq R_k,\ B_{k-1}+R_k-P_k,\ P_k$$ 
which is equivalent to $B_{k-1}\geq P_k\geq R_k$ which follows from  (A1).

 We shall also use  the standard fact that for $0<\alpha,\beta,\gamma,\delta<\infty$,
 $$
 	\| |f|^\alpha\|_{L^\beta}^\gamma=\| |f|^{\alpha\delta}\|_{L{^{\beta/\delta}}}^{\frac{\gamma}{\delta}},
 $$
 and set $\widetilde f(x)=f(-x)$.
 We compute
\begin{align*}
 &\|F\|_{L^{r_0,\bm r}}^{r_m}\\
 	&=\int_{\Rd}\Big(\int_{\Rd}\ldots \Big(\int_{\Rd}\Big(\int_{\Rd}
		| f_1(x-t_1)\ldots f_m(x-t_m)\, g(x)|^{r_0} 
			dx\Big)^\frac{r_1}{r_0}dt_1\Big)^\frac{r_2}{r_1}\ldots \Big)^{r_m}dt_m\\
 	&=\int_{\Rd}\Big(\int_{\Rd}\ldots \Big(\int_{\Rd}\Big(\int_{\Rd}
		| \widetilde {f_1}(t_1-x)\big(T_{t_2}\widetilde f_2(x)\ldots T_{t_m} \widetilde f_m(x)\, g(x)\big)|^{r_0} 
			dx\Big)^\frac{r_1}{r_0}dt_1\Big)^\frac{r_2}{r_1}\ldots \Big)^{r_m}dt_m\\
	&=\int_{\Rd}\Big(\int_{\Rd}\ldots \Big(\int_{\Rd}\Big(
		| \widetilde {f_1}|^{r_0}\ast |T_{t_2}\widetilde f_2\ldots T_{t_m}\widetilde f_m\, g|^{r_0} (t_1)
			\Big)^\frac{r_1}{r_0}dt_1\Big)^\frac{r_2}{r_1}\ldots \Big)^{r_m}dt_m\\
	&=\int_{\Rd}\Big(\int_{\Rd}\ldots\Big(\int_{\Rd} \Big\|
		| \widetilde {f_1}|^{r_0}\ast |T_{t_2}\widetilde f_2\ldots T_{t_m}\widetilde f_m\, g|^{r_0} 
			\Big\|_{L^{r_1/r_0}}^{\frac{r_1}{r_0}\frac{r_2}{r_1}}dt_2 \Big)^\frac{r_3}{r_2}\ldots \Big)^{r_m}					dt_m\\
	&\leq \int_{\Rd}\Big(\int_{\Rd}\ldots\Big(\int_{\Rd} \|
		 |\widetilde {f_1}|^{r_0}\|_{L^{p_1/r_0}}^{\frac{r_2}{r_0}} \ 
		 	\| |T_{t_2}\widetilde f_2\ldots T_{t_m}\widetilde f_m\, g|^{r_0}
		 	\|_{L^{b_1/r_0}}^{\frac{r_2}{r_0}}dt_2 \Big)^\frac{r_3}{r_2}\ldots \Big)^{r_m}
								dt_m\\
        &= \int_{\Rd}\Big(\int_{\Rd}\ldots\Big(\int_{\Rd} \|
		 \widetilde {f_1}\|_{L^{p_1}}^{r_2} \ \| |T_{t_2}f_2\ldots T_{t_m}f_m\, g|^{b_1}
		 	\|_{L^{1}}^{\frac{r_2}{ b_1}}dt_2 \Big)^\frac{r_3}{r_2}\ldots \Big)^{r_m}	
							dt_m\\
        &=\| {f_1}\|_{L^{p_1}}^{r_m} \int_{\Rd}\Big(\int_{\Rd}\ldots\Big(\int_{\Rd} \Big(\int_{\Rd}
		| f_2(x-t_2)\ldots f_m(x-t_m)\, g(x)|^{b_1} 
			dx\Big)^\frac{r_2}{b_1}dt_2\Big)^\frac{r_3}{r_2}\ldots \Big)^{r_m}	
							dt_m\\
	&\dots \\
	&\leq \| {f_1}\|_{L^{p_1}}^{r_m}\ldots \| {f_{m-1}}\|_{L^{p_{m-1}}}^{r_m} \int_{\Rd}\Big(\int_{\Rd}
		|  f_m(x-t_m)\, g(x)|^{b_{m-1}} 
			dx\Big)^\frac{r_{m}}{b_{m-1}} dt_m \\
	&=  \| {f_1}\|_{L^{p_1}}^{r_m}\ldots \| {f_{m-1}}\|_{L^{p_{m-1}}}^{r_m} \| |\widetilde{f}_m|^{b_{m-1}}
			\ast  |g|^{b_{m-1}}\|_{L^{\frac{r_m}{b_{m-1}}}}^{\frac{r_m}{b_{m-1}}}	\\
	&\leq \| {f_1}\|_{L^{p_1}}^{r_m}\ldots \| {f_{m-1}}\|_{L^{p_{m-1}}}^{r_m} \| f_m\|_{L^{p_m}}^{r_m}
			\|  |g|^{p_0}\|_{L^1}^{\frac{r_m}{p_0}}	\\
	&=\| {f_1}\|_{L^{p_1}}^{r_m}\ldots \| f_m\|_{L^{p_m}}^{r_m}
			\|  g\|_{L^{p_0}}^{r_m},				
\end{align*}
where each inequality  stems from an application of Young's inequality for
convolutions. In the final step, we used $b_m = p_0$ which follows by
combining the definition of $b_m$ with hypothesis (A3).
\end{proof}

\begin{remark} \label{remark:stacking}\rm

Observe that if we would add the condition $p_0\leq r_0$ in hypothesis (A1) of Lemma~\ref{lemma:Young1}, then (A1) and (A3) would combine to imply $p_k=r_k$ for $k=0,\ldots, m$.  Indeed, the strength of Lemma~\ref{lemma:Young1} lies in the fact that  $p_0\leq r_0$ and $p_k=r_k$ for $k=0,\ldots, m$ are not implied by the hypotheses. Setting $\Delta_k = \frac 1{p_k}-\frac 1 {r_k}$ for $k=0,\ldots,m$, (A1) in Lemma~\ref{lemma:Young1} is  $\Delta_1,\ldots, \Delta_m\geq 0$ and condition (A3) becomes $\Delta_0+\sum_{k=1}^m \Delta_k =0$, a condition that allows $\Delta_0$ to be negative, that is  $p_0 > r_0$. In short, all $\Delta_k>0$ contribute to compensate for $\Delta_0=r_0-p_0$ being negative.

Let us now  briefly discuss condition (A2) in Lemma~\ref{lemma:Young1}.  For $k=0$, we have $0\leq \frac 1 {r_0} -\frac 1 {p_1}$. To satisfy condition (A2) for $k=1$, we increase the left hand side by $\Delta_1=\frac 1{p_1} -\frac 1 {r_1}\geq 0$,  add to the right hand side the possibly negative term  $\frac 1{p_1} - \frac 1{p_2}$, and  require that the sum on the left remains bounded above by the sum on the right.  For $k=2$, we  increase  the left hand side by $\Delta_2=\frac 1{p_2} -\frac 1 {r_2}\geq 0$ and add to the right hand side   $\frac 1{p_2} - \frac 1{p_3}$, maintaining that  the right hand side dominates the left hand side. 
This is illustrated in Figure~\ref{fig:rectification} below.

In the case $m=1$,  the conditions $\Delta_1\geq 0$ and $\Delta_0+\Delta_1=0$ from Lemma~\ref{lemma:Young1} are amended  by the requirement $r_0 \leq p_1$, and, for example,  if $r_0=1$, $p_0=2$, then  Lemma~\ref{lemma:Young1} is applicable whenever $ 1\geq \frac 1 {p_1}= \frac 1 {r_1}+\frac 1 2$, that is, if $1\leq p_1= \frac{2r_1}{r_1+2}$.  

If $m=2$, then  $\Delta_1, \Delta_2 \geq 0$ and $\Delta_0+\Delta_1+\Delta_2=0$ from Lemma~\ref{lemma:Young1} are combined with the condition  $r_0 \leq p_1$ and $\Delta_1\leq  \frac 1 {r_0}-\frac 1 {p_2}$.  It is crucial in what follows to observe that these conditions are sensitive to the order of  the $p_k$ and the $r_k$.  For example, the parameters $r_0=1$, $p_0=2$, $r_1= 1= p_1$, $p_2=1$, $r_2=2$ satisfy the hypothesis, while $r_0=1$, $p_0=2$, $r_2= 1= p_2$, $p_1=1$, $r_1=2$ do not.

Indeed, if for some $k$,   $\Delta_k=\frac 1{p_k} -\frac 1 {r_k}$ is much smaller than  $\frac 1{p_{k}} - \frac 1{p_{k+1}}$, then we would profit  more from this if $k$ is a small index, that is, the respective summands play a role early on in the summation. 

Below, we shall use this idea and reorder the indices.  This allows us to first choose $\kappa(1)=k_1\in\{1,\ldots,d\}$ with $\Delta_{\kappa(1)}=\frac 1{p_{\kappa(1)}} -\frac 1 {r_{\kappa(1)}}$ small, and then $\kappa(2)=k_2$ so that $\frac 1{p_{\kappa(1)}} - \frac 1{p_{\kappa(2)}}$ is large.  Clearly, the feasibility of $\kappa(2)$ also depends on the size of  $\Delta_{\kappa(2)}=\frac 1{p_{\kappa(2)}} -\frac 1 {r_{\kappa(2)}}$, so finding an optimal order cannot be achieved with a greedy algorithm.  Moreover, note that the spaces $\mathcal M^{(p_0,\bf p),\kappa; (\bf q,q_{m+1}),\sigma}$ and $\mathcal M^{(p_0,\bf p),{\rm id}; (\bf q,q_{m+1}),{\rm id}}$ are not identical, hence, we cannot choose  $\kappa$ and $\rho$ arbitrarily.

\end{remark}

\begin{figure}[htbp]
  \centering
\begin{tikzpicture}[xscale=-1,yscale=1,rotate=90]

\begin{scope}[draw=blue!40, fill=blue!40]      
\filldraw (1,0) -- (1,1.4) -- (2,1.4) -- (2,0);
\end{scope}

\begin{scope}[draw=red!20, fill=red!20]      
\filldraw  (0,0) -- (1,0) -- (1,2.5) -- (0,2.5);
\end{scope}
\begin{scope}[draw=red!40, fill=red!40]      
\filldraw  (1,1.4) -- (2,1.4) -- (2,3.2) -- (1,3.2);
\end{scope}

\begin{scope}[draw=green!20, fill=green!20]      
\filldraw  (0,4.3) -- (1,4.3) -- (1,2.5) -- (0,2.5);
\end{scope}
\begin{scope}[draw=green!40, fill=green!40]      
\filldraw  (1,4.6) -- (2,4.6) -- (2,3.2) -- (1,3.2);
\end{scope}

\begin{scope}[draw=yellow!20, fill=yellow!20]      
\filldraw  (0,4.3) -- (1,4.3) -- (1,5.3) -- (0,5.3);
\end{scope}
\begin{scope}[draw=yellow!40, fill=yellow!40]      
\filldraw  (1,4.6) -- (2,4.6) -- (2,5.3) -- (1,5.3);
\end{scope}

\draw[->](0,0) -- (0,8.5);
\draw(1,0) -- (1,8.2);
\draw(2,0) -- (2,8.2);
\draw[blue](0,0) -- (1,0);
\draw(1,0) -- (2,0);
\draw(0,1) -- (2,1);
\draw(0,2) -- (2,2);
\draw(0,3) -- (2,3);
\draw(0,4) -- (2,4);
\draw(0,5) -- (2,5);
\draw(0,6) -- (2,6);
\draw(0,7) -- (2,7);
\draw(0,8) -- (2,8);

\node[below] at (0,0) {$0$};
\node[below] at (0,1) {$.5$};
\node[below] at (0,2) {$1$};
\node[below] at (0,3) {$1.5$};
\node[below] at (0,4) {$2$};
\node[below] at (0,5) {$2.5$};
\node[below] at (0,6) {$3$};
\node[below] at (0,7) {$3.5$};

\end{tikzpicture}
  \caption{Depiction of condition (A2) in Lemma~\ref{lemma:Young1}. After adding a pair of colored fields, the top row must always exceed the lower row, with the  lower row finally catching up in the last step, see Remark~\ref{remark:stacking}.}
  \label{fig:rectification}
\end{figure}
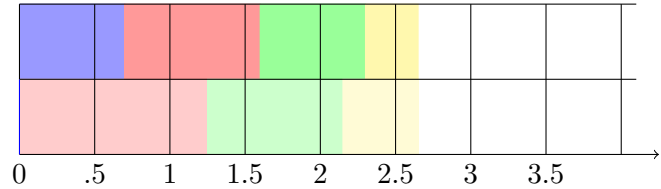

\begin{remark}\label{remark:increasing} \rm
Note that conditions {\rm (A1)} and {\rm (A3)} follow from (but are not equivalent to) the simpler condition
\begin{itemize} 
\item[(A4)] \label{(A4)} $\displaystyle 1\leq r_0\leq p_1 \leq r_1 \leq p_2 \leq \ldots \leq r_{m-1} \leq p_{m}\leq r_m \leq \infty$. 
\end{itemize} Equality {\rm (A3)} can then be satisfied by choosing an appropriate $p_0\geq 1$.

The inequalities in (A1) imply that the LHS of (A2) is positive and, hence, always $r_0\leq p_k\leq r_k$ for all $k$.
Also, (A1) and (A2) necessitate $p_k\leq r_k\leq p_{m+1}$. 

 \end{remark}

Similarly to Lemma~\ref{lemma:Young1}, we show the following.

\begin{lemma}\label{lemma:Young2}  
 Suppose that $1\leq q_k, s_k \leq \infty$ for $k= 1, \ldots, m+1$ and
\begin{itemize}
\item[(B1)] $q_k\geq s_k$, $k=1,\ldots,m$;
\item[(B2)] $\displaystyle  \sum_{\ell =k+1}^{m} \frac 1{q_\ell} - \frac 1{s_\ell} \geq  \frac 1{s_{m+1}} -  \frac 1 {q_{k}} , \quad k=1,\ldots, m$;
\item[(B3)] $\displaystyle\sum_{\ell =1}^m \frac 1 { q_\ell} - \frac 1 {  s_\ell}=  \frac 1 {s_{m+1}}-\frac 1 {q_{m+1}}$;
\end{itemize}
then for 
$G(\bm t,x)= \bm f (\bm t) g(x+S(\bm t))$ we have
\begin{align*}
 \|G\|_{L^{\bm s,s_{m+1}}} \leq  \|\bm f\|_{L^{\bm q}} \ \|g\|_{L^{q_{m+1}}}.
\end{align*}
\end{lemma}

\begin{proof}
As before, our computations involve the introduction of an auxiliary parameter $b_k$.  We start with a formal computation, namely,
\begin{align*}
 &\|G\|_{L^{ \bm r, s_{m+1}}}^{s_{m+1}}\\
 	&=\int_{\Rd}\Big(\int_{\Rd}\ldots 
	\Big(\int_{\Rd}
		| f_1(t_1)\ldots f_m(t_m)\, g(x+t_1+\ldots+t_m)|^{s_1} 
			dt_1\Big)^\frac{s_2}{s_1} \ldots 
			dt_m\Big)^\frac{s_{m+1}}{s_m}dx\\
 	&=\int_{\Rd}\Big(\int_{\Rd}|\widetilde f_m(t_m)|^{s_m}\Big(\ldots \int_{\Rd}|\widetilde f_2(t_2)|^{s_2}\Big(\int_{\Rd}
		|  \widetilde f_1(t_1)  g(x-t_1{-} t_2{-} \ldots {-}t_m)|^{s_1} 
			dt_1\Big)^\frac{s_2}{s_1}dt_2\Big)^\frac{s_3}{s_2}\ldots \Big)^{\frac{s_{m+1}}{s_m}}dx\\
	&=\int_{\Rd}\Big(\int_{\Rd}|\widetilde f_m(t_m)|^{s_m}\Big( \ldots \int_{\Rd}|\widetilde f_2(t_{2})|^{s_2}	 	
		\big(|  \widetilde f_1|^{s_1} \ast |g|^{s_1}(x - t_{2} {-}t_3{-} \ldots {-}t_m) 
			\big)^\frac{s_2}{s_1}dt_{2}\Big)^\frac{s_3}{s_2}dt_3\Big)^\frac{s_4}{s_3}\ldots \Big)^{\frac{s_{m+1}}{s_m}}dx\\
	&\ldots \\
	&= \int_{\Rd}\Big( |\widetilde f_m|^{s_m}\ast \Big( |\widetilde f_{m-1}|^{s_{m-1}}\ast \Big(\ldots \Big( |\widetilde f_2|^{s_2}	\ast 	
		\big(|  \widetilde f_1|^{s_1} \ast | g(x)|^{s_1} 
			\big)^\frac{s_2}{s_1}\Big)^\frac{s_3}{s_2}\ldots \Big)^\frac{s_m}{s_{m-1}}\Big)^{\frac{s_{m+1}}{s_m}}dx\\
	&= \Big\| |\widetilde f_m|^{s_m}\ast \Big( |\widetilde f_{m-1}|^{s_{m-1}}\ast \Big(\ldots \Big( |\widetilde f_2|^{s_2}	\ast 	
		\big(|  \widetilde f_1|^{s_1} \ast | g|^{s_1} 
			\big)^\frac{s_2}{s_1}\Big)^\frac{s_3}{s_2}\ldots \Big)^\frac{s_m}{s_{m-1}}\Big\|_{L^{\frac{s_{m+1}}{s_m}} }^{\frac{s_{m+1}}{s_m}}  \\
&\leq \| | \widetilde f_m|^{s_m}\|_{L^{q_m/s_m}}^{\frac{s_{m+1}}{s_m}}  \ 
		\Big\| \Big(|\widetilde f_{m-1}|^{s_{m-1}} 
			\ast \Big(\ldots \Big( |\widetilde f_2|^{s_2}	
				\ast 	\big(|  \widetilde f_1|^{s_1} 
					\ast | g|^{s_1} \big)^\frac{s_2}{s_1}\Big)^\frac{s_3}{s_2}\ldots 
						\Big)^{\frac{s_m}{s_{m-1}}} \|_{L^{b_m/s_m} }^{\frac{s_{m+1}}{s_m}}\\
&= \|   f_m\|_{L^{q_m}}^{s_{m+1}} \ 
		\Big\| |\widetilde f_{m-1}|^{s_{m-1}} 
			\ast \Big(\ldots \Big( |\widetilde f_2|^{s_2}	
				\ast 	\big(|  \widetilde f_1|^{s_1} 
					\ast | g|^{s_1} \big)^\frac{s_2}{s_1}\Big)^\frac{s_3}{s_2}\ldots 
						\Big)^{\frac{s_{m-1}}{s_{m-2}}}  \|_{L^{b_m/s_{m-1}} }^{\frac{s_{m+1}}{s_{m-1}}} \\
& \ldots \\
&= \|   f_m\|_{L^{q_m}}^{s_{m+1}} \ldots \|   f_2 \|_{L^{q_2}}^{s_{m+1}}
		\Big\| |  \widetilde f_1|^{s_1} 
					\ast | g|^{s_1}   \|_{L^{b_2/s_1} }^{\frac{s_{m+1}}{s_{1}}} \\
&\leq \|   f_m\|_{L^{q_m}}^{s_{m+1}} \ldots \|   f_2 \|_{L^{q_2}}^{s_{m+1}}
		\| |  \widetilde f_1|^{s_1} \|_{L^{q_1/s_1} }^{\frac{s_{m+1}}{s_{1}}}
					\| | g|^{s_1}   \|_{L^{q_{m+1}/s_1} }^{\frac{s_{m+1}}{s_{1}}} \\
&= \|   f_m\|_{L^{q_m}}^{s_{m+1}}\ldots  \|   f_1\|_{L^{q_1}}^{s_{m+1}}  \|  g\|_{L^{q_{m+1}}}^{s_{m+1}} .
			\end{align*}
			
To justify the first application of Young's inequality, we require 
$$\frac{1}{\tfrac{q_m}{s_m}}+ \frac{1}{\tfrac{b_m}{s_m}} =1+\frac{1}{\tfrac{s_{m+1}}{s_m}},\quad  \frac{q_m}{s_m}, \,\frac{b_m}{s_m},\, \frac{s_{m+1}}{s_m}\geq 1.$$
Using reciprocals, this is equivalent to 
$$ Q_m + B_m= S_m +S_{m+1} ,\quad  S_m \geq Q_m,\, B_m,\, S_{m+1},$$
that is,
$$  B_m= S_m -Q_m+S_{m+1} ,\quad  S_m \geq Q_m,\, B_m,\, S_{m+1}.$$
The subsequent application of Young's inequality requires
$$ \frac{1}{\frac{q_{m-1}}{s_{m-1}}} + \frac{1}{\frac{b_{m-1}}{s_{m-1}}} =1+\frac{1}{\frac{b_{m}}{s_{m-1}}},\quad   \frac{q_{m-1}}{s_{m-1}},\,\frac{b_{m-1}}{s_{m-1}}, \, \frac{b_m}{s_{m-1}}   \geq 1.$$
Using reciprocals, this is equivalent to 
$$  B_{m-1}= S_{m-1}- Q_{m-1}+B_m =S_{m+1}+ \sum_{\ell=m-1}^m S_\ell-Q_\ell,\quad  S_{m-1} \geq Q_{m-1},\, B_{m-1},\, B_{m}.$$

In general, for $k=1,\ldots,m-2$, we require
$$B_{m-k}=S_{m-k}-Q_{m-k}+B_{m-k+1}=S_{m+1}+ \sum_{\ell=m-k}^m S_\ell-Q_\ell,\quad S_{m-k}\geq Q_{m-k},B_{m-k},B_{m-k+1},$$
and finally, for the last application of Young's inequality, we require
$$ Q_{m+1}=S_{1}-Q_1+B_{2} = S_{m+1}+ \sum_{\ell=m-k}^m S_\ell-Q_\ell,\quad S_{1}\geq Q_{1},Q_{m+1},B_{2}.$$ 

Now, $S_k\geq Q_k$ for $k=1,\ldots,m$ implies 
$$0\leq S_{m+1}\leq B_m \leq B_{m-1} \leq \ldots \leq B_3\leq B_2\leq Q_{m+1},$$ hence, it suffices to postulate aside of $S_k\geq Q_k$ for $k=1,\ldots,m$ the conditions $S_k\geq B_k$ for $k=2,\ldots, m$ and  $S_1\geq Q_{m+1},B_2$. For $k=2,\dots,m$, we use that  $\sum_{\ell=1}^{m+1} S_\ell-Q_\ell =0$ implies $\sum_{\ell=k}^m S_\ell-Q_\ell = -S_{m+1}+Q_{m+1}-  \sum_{\ell=1}^{k-1} S_\ell-Q_\ell $ in order to rewrite  $S_k\geq B_k$ in form of  
\begin{align*}
 S_k &\geq B_k  =S_{m+1}+ \sum_{\ell=k}^m S_\ell-Q_\ell =   Q_{m+1} - \sum_{\ell=1}^{k-1} S_\ell-Q_\ell 
\end{align*}
which is
\begin{align*}
 Q_{m+1} - S_k \leq \sum_{\ell=1}^{k-1} S_\ell-Q_\ell.   
\end{align*}
For $k=1$, the above covers the condition $Q_{m+1}\leq S_1$.

In summary, for $k= 1, \ldots, m+1$ we obtained the sufficient conditions
\begin{itemize}
\item[(B1')] $q_k\geq s_k$, $k=1,\ldots,m$
\item[(B2')] $\displaystyle  \sum_{\ell =1}^{k} \frac 1{s_\ell} - \frac 1{q_\ell} \geq  \frac 1{q_{m+1}} - \frac 1 {s_{k+1}} , \quad k=0,\ldots, m-1$;
\item[(B3')] $\displaystyle\sum_{\ell =1}^m \frac 1 { s_\ell} - \frac 1 {  q_\ell}=  \frac 1 {q_{m+1}}-\frac 1 {s_{m+1}}$.
\end{itemize}
Forming the difference of (B3') and (B2') gives
\begin{itemize}
\item[(B2'')] $\displaystyle  \sum_{\ell =k+1}^{m} \frac 1{s_\ell} - \frac 1{q_\ell} \leq  \frac 1 {s_{k+1}} - \frac 1{s_{m+1}} , \quad k=0,\ldots, m-1$.
\end{itemize}
Reindexing   leads to
\begin{itemize}
\item[(B2'')] $\displaystyle  \sum_{\ell =k}^{m} \frac 1{s_\ell} - \frac 1{q_\ell} \leq  \frac 1 {s_{k}} - \frac 1{s_{m+1}} , \quad k=1,\ldots, m$,
\end{itemize}
and adding $\frac 1 {q_k}-\frac 1 {s_{k}}$ to both sides, and then multiplying both sides by -1 gives
\begin{itemize}
\item[(B2)] $\displaystyle  \sum_{\ell =k+1}^{m} \frac 1{q_\ell} - \frac 1{s_\ell} \geq  \frac 1{s_{m+1}} -  \frac 1 {q_{k}} , \quad k=1,\ldots, m$.
\end{itemize}

\end{proof}
\begin{remark}\rm
The conditions {\rm (B1)--(B3)} are similar to those in  {\rm (A1)--(A3)}. Indeed, a change of variable $k\to m+1-k$, that is, renaming $q_k= \widetilde{q}_{m+1-k}$ and $s_k= \widetilde{s}_{m+1-k}$, $k=1,\ldots,m+1$, turns {\rm (B2)} into
$$\displaystyle  \sum_{\ell =k+1}^{m} \frac 1{\widetilde{q}_{m+1-\ell}} - \frac 1{\widetilde{s}_{m+1-\ell}} \geq  \frac 1{\widetilde{s}_{m+1-(m+1)}} -  \frac 1 {\widetilde{q}_{m+1-k}}= \frac 1{\widetilde{s}_{0}} -  \frac 1 {\widetilde{q}_{m+1-k}} , \quad k=1,\ldots, m.$$ 
We have  $$\displaystyle  \sum_{\ell =k+1}^{m} \frac 1{\widetilde{q}_{m+1-\ell}} - \frac 1{\widetilde{s}_{m+1-\ell}} =  \sum_{\ell'=1}^{m-k} \frac 1{\widetilde{q}_{\ell'}} - \frac 1{\widetilde{s}_{\ell'}},$$
hence, we obtain for $k'=m-k$ the conditions
$$\displaystyle  \sum_{\ell' =1}^{k'} \frac 1{\widetilde{q}_{\ell'}} - \frac 1{\widetilde{s}_{\ell'}} \geq  \frac 1{\widetilde{s}_{0}} -  \frac 1 {\widetilde{q}_{k'+1}} , \quad k'=0,\ldots, m-1. $$
We conclude that  difference between the conditions in Lemma~\ref{lemma:Young1} and in Lemma~\ref{lemma:Young2} lies --- aside of naming the decay parameters --- simply in replacing  $\leq$ in  {\rm (A1)} and {\rm (A2)} by $\geq$ in  {\rm (B1)} and {\rm (B2)}. Hence, it comes to no surprise that   {\rm (B1)} and {\rm (B2)} follow from, but are not equivalent to
\begin{itemize} 
\item[(B4)] $\displaystyle 1\leq s_1 \leq q_1 \leq s_2 \leq \ldots \leq q_{m-1} \leq s_{m}\leq q_m \leq s_{m+1}\leq \infty$. 
\end{itemize}

Moreover, (B1) implies $\displaystyle  \sum_{\ell =k+1}^{m} \frac 1{q_\ell} - \frac 1{s_\ell} \leq  0$, and, hence, $ q_{m+1} \geq q_{k}$ for $ k=1,\ldots, m$.
\end{remark}

\subsection{Young type results with permutations}\label{sec:young permute}

As observed in Remark~\ref{remark:stacking}, condition (A2) in Lemma~\ref{lemma:Young1} and, similarly, (B2) in Lemma~\ref{lemma:Young2}
are sensitive to the order of the $p_k$, $r_k$, $q_k$, and $s_k$. 

To
obtain a bound for operators as desired, we may have to reorder the
parameters.  This motivates the introduction of permutations $\kappa$ and
$\rho$. In addition to the flexibility obtained at cost of notational complexity, we observe that the permutation of the integration order will allow us to pull out integration with respect to some variables.  In fact, setting $t_0=x$ and choosing $j=\kappa^{-1}(0)$, we arrive at
\begin{align*}
 &\|F\|_{L^{(r_0,\bm r);\kappa}}^{r_{\k(m)}}\\
 	&=\int_{\Rd}\Big(\int_{\Rd}\ldots \Big(\int_{\Rd}\Big(\int_{\Rd}
		| f_1(t_0-t_1)\ldots f_m(t_0-t_m)\, g(t_0)|^{r_{\k(0)}} 
			d{t_{\k(0)}}\Big)^\frac{r_{\k(1)}}{r_{\k(0)}}dt_{\k(1)}\Big)^\frac{r_{\k(2)}}{r_{\k(1)}}\ldots \Big)^{r_{\k(m)}}dt_{\k(m)}\\
&= \int_{\Rd}\Big(\int_{\Rd}\ldots \Big(\int_{\Rd}\Big(\int_{\Rd} \prod_{\ell=0}^{m}|f_{\kappa(\ell)}(t_{0}-t_{\kappa(\ell)}) g(t_0)|^{r_{\k(0)}} 
			d{t_{\k(0)}}\Big)^\frac{r_{\k(1)}}{r_{\k(0)}}dt_{\k(1)}\Big)^\frac{r_{\k(2)}}{r_{\k(1)}}\ldots \Big)^{r_{\k(m)}}dt_{\k(m)}\\
&=\|f_{\k(0)}\|^{r_{\k(m)}}_{L^{p_{\k(0)}}}\|f_{\k(1)}\|^{r_{\k(m)}}_{L^{p_{\k(1)}}}\ldots\|f_{\k(j-1)}\|^{r_{\k(m)}}_{L^{p_{\k(j-1)}}}\times\\
            &\int_{}\Big(\int_{}\ldots \Big(\int_{}\Big(\int_{}
		| f_{\k(j+1)}(x-t_{\k(j+1)})\ldots f_{\k(m)}(x-t_{\k(m)})\, g(x)|^{r_{\k(j)}} 
			d{x}\Big)^\frac{r_{\k(j+1)}}{r_{\k(j)}} dt_{\k(j+1)}\Big)^\frac{r_{\k(j+2)}}{r_{\k(j+1)}}\ldots \Big)^{r_{\k(m)}}dt_{\k(m)}.
 \end{align*} We can then apply Lemma~\ref{lemma:Young1} to the iterated integral on the right hand side.

This observation leads us  to the following result.

\begin{lemma}\label{lemma:Young1b}  
 Let $\kappa$ be a permutation on $\{0,1,\ldots,m\}$, $z=\kappa^{-1}(0)$, and let $1\leq p_k, r_k \leq 
  \infty$,  $k=0, 1, \ldots, m$, satisfy
\begin{itemize}
\item[(A0)] $p_{\k(\ell)}=r_{\k(\ell)}$, $\ell=0,\ldots,z-1$;
\item[(A1)] $  p_{\k(\ell)}\leq r_{\k(\ell)}$, $\ell=z,\ldots,m$; 
\item[(A2)]   $\displaystyle \sum_{\ell =z+1}^{k}  \frac 1{p_{\k(\ell)}} - \frac 1{ r_{\k(\ell)}} \leq \frac 1 {r_{0}} - \frac 1{p_{\k(k+1)}} , \quad k=z,\ldots, m-1$;
\item[(A3)] $\displaystyle \sum_{\ell =z+1}^{m} \frac 1 { p_{\kappa(\ell)}} - \frac 1 {  r_{\kappa(\ell)}} = \frac 1 {r_{{0}}}-\frac 1 {p_{0}}  $.
\end{itemize}
 Then for
   $F(x,\bm t)=\bm f (x- \bm t) g(x)$ it holds
 \begin{align*}
 \|F\|_{L^{(r_0,\bm r)_\kappa}}\leq   \|g\|_{L^{p_0}} \ \|\bm f\|_{L^{\bm p}}.
 \end{align*}

%
%
%

\end{lemma}

\begin{remark} \label{remark:Minkowski}
 \rm Loosely speaking, the decay of a function  $F(x,t_1,\ldots,t_d)$ in the variables $(x,t_1,\ldots,t_d)$ is given by the parameters $(p_0,p_1,\ldots,p_d)$, that is, $L^{p_0}$-decay in $x$,   $L^{p_1}$-decay in $t_1$, $\ldots$, $L^{p_d}$-decay in $t_d$. As we then use the flexibility of order of integration, it is worth noting that Minkowski's inequality for integrals implies that integrating with respect to variables with large exponents last,  increases the size of the space.
 
 For example, if $q\geq p$, we have 
\begin{align*}
  \|F\|_{L^{(p,q);(0,1)}}&= \Big( \int \Big(\int  |F(x,t_1)|^p dx\Big)^{q/p}dt_1\Big)^{1/q}\leq \Big( \int \Big(\int  |F(x,t_1)|^q dx\Big)^{p/q}dt_1\Big)^{1/p}= \|F\|_{L^{(p,q);(1,0)}},
\end{align*}
which implies $L^{(p,q);(0,1)}\subseteq L^{(p,q);(1,0)}$ if $q \geq p$, for example,  $L^{(1,\infty);(0,1)}\subseteq L^{(1,\infty);(1,0)}$. This inclusion is strict in general, for example, choose $F(x,t_1)=g(x-t_1)\in L^{(1,\infty);(1,0)} \setminus L^{(1,\infty);(0,1)}$ for any function $g\in L^1$. 
\end{remark}

Similarly to Lemma~\ref{lemma:Young1b}, we formulate the following.


\begin{lemma}\label{lemma:Young2b}  
Let $\rho$ be a permutation on $\{1,\ldots,m+1\}$, $w=\rho^{-1}(m+1)$, and
  $1\leq q_k,s_k \leq
  \infty$ be $k= 1, \ldots, m+1$ satisfy
\begin{itemize}
\item[(B0)] $q_{\k(\ell)}=s_{\k(\ell)}$, $\ell=w,\ldots,m$;
\item[(B1)] $q_{\rho(k)}\geq s_{\rho(k)}$, $k=1,\ldots,w-1$;
\item[(B2)] $\displaystyle  \sum_{\ell =k+1}^{w-1} \frac 1{q_{\rho(\ell)}} - \frac 1{s_{\rho(\ell)}} \geq  \frac 1{s_{m+1}} -  \frac 1 {q_{\rho(k)}} , \quad k=1,\ldots, w-1$

\item[(B3)] $\displaystyle \sum_{\ell =1}^{w-1}  \frac 1 { q_{\rho(\ell)}} - \frac 1 {  s_{\rho(\ell)}} =\frac 1 {s_{m+1}}-\frac 1 {q_{m+1}}$.
\end{itemize}
 Then  $G(\bm \xi,\nu)= \bm f (\bm \xi) g(\nu+S(\bm \xi))$ satisfies
\begin{align*}
 \|G\|_{L^{(\bm s,s_{m+1}),\rho}} \leq  \|\bm f\|_{L^{{\bm q}}} \ \|g\|_{L^{q_{m+1}}} \, .
\end{align*}
\end{lemma}

\section{Boundedness on modulation spaces}\label{sec3}
When applying Lemmas~\ref{lemma:Young1}, ~\ref{lemma:Young2}, \ref{lemma:Young1b}, and \ref{lemma:Young2b}  in the context of modulation spaces, we can use the property that  $M^{p_{1},q_{1}}$ embeds continuously in $M^{p_{2},q_{2}}$ if $p_1\leq p_2$ and $q_1\leq q_2$. To exploit this in full,  the introduction of auxiliary parameters $\widetilde {\bm p}$ and $\widetilde {\bm s}$ is required as illustrated by Example~\ref{example:auxillary}  below.


\begin{proposition}\label{Vtilde}
Given $1\leq p_0, \bm p, \widetilde{\bm p},  \bm q, \widetilde{\bm
  q}, q_{m+1}, r_0, \bm r, \bm s, s_{m+1} \leq \infty$ with $\bm p\leq\widetilde{\bm
  p}\leq  \bm r $ and $ \bm s, \bm q \leq\widetilde{\bm q} $. Let
$\kappa$  be a permutation on $\{0,\ldots,m\}$ and  let $z=\kappa^{-1}(0)$. Similarly, let  $\rho$ be a permutation on  $\{1,2,\ldots,m+1\}$  and  $w=\rho^{-1}(m+1)$. Assume
 \begin{enumerate}
\item $\displaystyle  \sum_{\ell =z+1}^{k} \frac 1{  \widetilde p_{\kappa(\ell)} } - \frac 1 {  r_{\kappa(\ell)}} \leq \frac 1{ r_0} - \frac 1 {\widetilde p_{\kappa(k+1)}}    , \quad k=z,\ldots, m-1$;
\item $\displaystyle \sum_{\ell =z+1}^m \frac 1 { \widetilde p_{\kappa(\ell)} } - \frac 1  {r_{\kappa(\ell)}} \geq \frac 1{ r_0} - \frac 1 {p_0}  $;
\item $\displaystyle  \sum_{\ell =k+1}^{w-1} \frac 1{\widetilde{q}_{\rho(\ell)}}  - \frac 1 {   {  s_{\rho(\ell)}}  }\geq \frac 1{  {s}_{m+1}} - \frac 1 { \widetilde {q}_{k}}    , \quad k=1,\ldots, w-1$;
\item \  $\displaystyle \sum_{\ell =1}^{w-1}\ \frac 1{\widetilde{q}_{\rho(\ell)}}  - \frac 1 {   {  s_{\rho(\ell)}}   }\geq \frac 1{ s_{m+1}} - \frac 1 {q_{m+1}} .$
\end{enumerate}
Let $v$
be a weight function 
on $\RR^{2(m+1)d}$  and assume that $w_0,w_1,\ldots, w_m$ are weights on $\Rr^{2d}$ such that
 \begin{equation}\label{WeightCondition}
 v(x,\bm t,\bm \xi,\nu) \leq   w_{0}(x, \nu +
S(\xi)) w_1(x-  t_1,  \xi_1)\cdot \ldots \cdot w_m(x-  t_m,  \xi_m).
\end{equation} 
For 
$\varphi \in S(\Rd)$ real valued, $\bm f \in M^{\bm p,{\kappa}; \bm q,{\rho}}_{\bm w}(\Rr^{md})$, and $g\in M^{p_{0},q_{m+1}}_{w_0}(\Rr^{d})$, we have
$\mathcal{V}_{T_A(\bm \varphi\otimes\varphi)}T_A(\overline{\bm
  f}\otimes g) \in L_{v}^{(r_0,\bm r)_\kappa, (\bm s, s_{m+1})_\rho}(\Rr^{2(m+1)d})$ with
\begin{align} \label{eqn:mainproposition}
 \|\mathcal{V}_{T_A(\bm \varphi\otimes\varphi)}T_A(\overline{\bm f}\otimes g)\|_{L_{v}^{(r_0,\bm r),\kappa; (\bm s, s_{m+1}),\rho}}\leq C\,
\|f_{1}\|_{M_{w_1}^{p_1,q_1}}\ldots \|f_m\|_{M_{w_m}^{p_m,q_m}} \|g\|_{M_{w_0}^{p_0,q_{m+1}}},
\end{align}
where the LHS is defined by integrating the variables in the index order 
$$
\kappa(0),\kappa(1),\ldots,\kappa(m),\rho(1),\ldots,\rho(m), \rho(m+1).
$$
In particular, $$\|T_A(\overline{\bm f}\otimes g)\|_{\mathcal{M}_{v}^{(r_0,\bm r)_\kappa, (\bm s, s_{m+1})_\rho}} \leq C\,
\|f_{1}\|_{M_{w_1}^{p_1,q_1}}\ldots \|f_{m}\|_{M_{w_m}^{p_m,q_m}} \|g\|_{M_{w_0}^{p_0,q_{m+1}}}.$$
Note that  $C$ depends only on the parameters $p_i$, $r_i$, $q_i$, $s_i$ and $d$.
\end{proposition}

\begin{proof} For simplicity we assume $\rho=\kappa=id$ and use Lemma~\ref{lemma:Young1} and  Lemma~\ref{lemma:Young2}. The general case follows as   Lemma~\ref{lemma:Young1b} and  Lemma~\ref{lemma:Young2b} followed from Lemma~\ref{lemma:Young1} and  Lemma~\ref{lemma:Young2}.

Let $\bm f=(f_1,f_2,\ldots,f_m)$,  $\bm\phi=(\phi,\phi,\ldots,\phi)$ and $\bm w=(w_1.w_2.\ldots.w_m)$. Then
$$V_{\bm\phi}\bm f=V_\phi f_1\otimes\ V_\phi f_2\otimes\dots \otimes V_\phi f_m,$$
and by Lemma  \ref{lemmaT_A}, we have
$$
	\mathcal V_{T_A(\bm\varphi\otimes\varphi)}T_A(\overline{\bm f}\otimes g)(x,-\bm \xi,\bm t, \nu)=
\overline{  V_{\bm\varphi}\bm  f (x-\bm t, \bm\xi)} V_{\varphi}g (x,\nu-S(\bm \xi)),
 $$
where $x,\nu \in \Rd$, and $\bm t, \bm \xi\in\RR^{md}$.

So, if~\eqref{WeightCondition} and conditions (2) and (4) above hold with equality, then 
(A1)--(A3) in Lemma~\ref{lemma:Young1} and (B1)--(B3)  in Lemma~\ref{lemma:Young2} will hold. Then
\begin{align}
&\|v(x,\bm t,\bm \xi,\nu)\mathcal{\bm V}_{T_A(\bm \varphi\otimes\varphi)}T_A(\overline{\bm f}\otimes g)(x,\bm t,\bm \xi,\nu)\|_{L^{(r_0,\bm r), (\bm s, s_{m+1})}(x,\bm t,\bm \xi,\nu)}\nonumber\\
&\leq\big\| \| \bm w(x-\bm t,\xi)\left(V_{\bm \varphi}\bm f\right)(x-\bm t,\bm \xi) w_0 (x,\nu+S(\bm \xi))\left(V_{\varphi}g\right)(x,\nu+S(\bm \xi)) \|_{L^{r_0,\bm r}(x,\bm t)}\big \|_{L^{\bm s, s_{m+1} }(\bm \xi,\nu)} \nonumber\\
&\leq\big\|\| \bm w(\bm t,\xi)\left(V_{\bm \varphi}\bm f\right)(\bm t,\bm \xi)\|_{L^{\bm p}(\bm t)} \|w_0(x,\nu+S(\bm \xi))\left(V_{\varphi}g\right)(x,\nu+S(\bm \xi))\|_{L^{p_0}(x)}\big \|_{L^{\bm s, s_{m+1} }(\bm \xi,\nu)} \nonumber\\
&\leq \big\|\| \bm w(\bm t,\xi)\left(V_{\bm \varphi}\bm f\right)(\bm t,\bm \xi)\|_{L^{\bm p}(\bm t)} \big \|_{L^{\bm q}(\bm \xi)} \nonumber\\
&\qquad \qquad \qquad \big\| \|w_0(x,\nu+S(\bm \xi))\left(V_{\varphi}g\right)(x,\nu+S(\bm \xi))\|_{L^{p_0}(x)} \big \|_{L^{ q_{m+1} }(\nu)}  \nonumber\\
&= \| V_{\bm \varphi}\bm f\|_{L_{\bm w}^{\bm p,\bm q} }  
\   \| V_{\varphi}g \|_{L_{w_0}^{p_0  q_{m+1} } } \nonumber.
\end{align}

We now use  that $p\leq \widetilde p$ and $q\leq \widetilde q$ implies
$
 \|f\|_{M^{\widetilde p,\widetilde q}} \lesssim \|f\|_{M^{ p, q}},
$
a property that clearly carries through to the class of weighted modulation spaces considered in this paper. 
If hypotheses (2) and (4)  hold with strict inequalities, then we can increase $p_0$ to appropriate $\widetilde p_0$ and $q_{m+1}$ to appropriate $\widetilde q_{m+1}$ so that (2) and (4) will hold with equalities. 
The resulting inequalities involving $\widetilde p_0$ and $\widetilde q_{m+1}$ then again implies the weaker inequalities involving $ p_0$ and $q_{m+1}$.

\end{proof}

\begin{example}\label{example:auxillary} \rm
The conditions $r_0\leq p_k\leq r_k$, $k=1,\ldots,m$, and
   $\sum_{\ell=1}^m1/p_\ell-1/r_\ell \geq 1/r_0-1/p_0$ do not
   guarantee the existence of a permutation $\kappa$ so that also
   $\sum_{\ell=1}^k 1/p_{\kappa(\ell)}-1/r_{\kappa(\ell)} \geq
   1/r_0-1/p_{\kappa(k+1)}$, $k=0,\ldots,m-1$.  Indeed, consider for $m=2$ the case $r_0=1$, $p_1=p_2=10/9$, $r_1=r_2=2$, and $p_0=5$. It is easy to see that no $\kappa$ exist that allows us to apply Proposition~\ref{Vtilde} to obtain for these parameters \eqref{eqn:mainproposition}.  Using $r_0=1$, $p_1=p_2=10/9$, $r_1=r_2=2$, $r_1=r_2=2$, we can choose $\kappa(0)=1$, $\kappa(1)=0$, $\kappa(2)=2$, to obtain \eqref{eqn:mainproposition} for $p_0 \leq 5/3$.

Unfortunately, this is again not the best we can do.  In fact, we can replace $p_2$ by $\widetilde p_2=15/9\in [10/9,18/9]=[p_2,r_2]$.  This choice allows us to choose for $\kappa$ the identity which leads to sufficiency for $p_0\leq 2$, which by inclusion also gives boundedness with $r_0=1$, $p_1=p_2=10/9$, $r_1=r_2=2$, $r_1=r_2=2$.   
  \end{example}

\begin{remark}\rm Observe those $k$ with $r_0> r_k$ must satisfy $\kappa^{-1}(k)< z$;  possibly there are also $k$ with $r_0\leq r_k$ and $\kappa^{-1}(k)< z$.
Importantly, only those $k$ with $p_k < r_k$ and $\kappa^{-1}(k)>z$ contribute to filling the gap between $p_0$ and $r_0$, see Remark~\ref{remark:stacking}
 \end{remark}

As immediate consequence, we obtain our first main result.

\begin{theorem}\label{thm:main} 
Given $1\leq p_0, \bm p, \widetilde{\bm p},  \bm q,\widetilde{\bm
  q}, q_{m+1}, r_0, \bm r, \bm s, s_{m+1} \leq \infty$ with $\bm p\leq\widetilde{\bm
  p}\leq  \bm r' $ and $\bm q,\bm s' \leq \widetilde{\bm q} $. Let
$\kappa$  be a permutation on $\{0,\ldots,m\}$ and  let $z=\kappa^{-1}(0)$. Similarly, let  $\rho$ be a permutation on  $\{1,2,\ldots,m+1\}$  and  $w=\rho^{-1}(m+1)$ and
 \begin{enumerate}
\item $\displaystyle  \frac 1{ r_0} + \frac 1 {\widetilde p_{\kappa(k+1)}}+\sum_{\ell =z+1}^{k}  \frac 1{\widetilde p_{\kappa(\ell)}} + \frac 1 {  r_{\kappa(\ell)}} \leq k{-}z{+}1,  \ k=z,\ldots, m{-}1$;
\item $\displaystyle \frac 1 {r_0}+\sum_{\ell =z+1}^m \frac 1{\widetilde p_{\kappa(\ell)}}  + \frac 1 {  r_{\kappa(\ell)}}\geq m-z + \frac 1 {p_0} $;
\item $\displaystyle \frac 1 { s_{m+1}} + \frac 1{ \widetilde{q}_{\rho(k)}} +  \sum_{\ell =k+1}^{w-1} \frac 1{\widetilde{q}_{\rho(\ell)}}  + \frac 1   {  s_{\rho(\ell)}}\geq  w-k     , \quad k=1,\ldots, w-1$;
\item $\displaystyle  \frac 1 {s_{m+1}}+ \sum_{\ell =1}^{w-1} \frac 1{\widetilde{q}_{\rho(\ell)}}  + \frac 1 {   s_{\rho(\ell)} }\geq w-1 +\frac 1{ q_{m+1}}  $.
\end{enumerate}
Let $v$
be a weight function 
on $\RR^{2(m+1)d}$  and assume that $w_0,w_1,\ldots, w_m$ are weights on $\Rr^{2d}$ such that
\begin{equation}\nonumber
 v(x,\bm t,-\bm \xi,\nu)^{-1}\leq   w_{0}(x, \nu +
S(\xi))^{-1}w_1(x-  t_1,  \xi_1)\cdot \ldots \cdot w_m(x-  t_m,  \xi_m).
\end{equation} 

Assume that $\sigma \in \MV.$  Then the multilinear pseudodifferential operator $T_\sigma$ defined
initially for  $f_k \in S(\Rd)$ for $k=1,
2, \hdots, m$ by ~\eqref{mpsido} extends to a bounded multilinear
operator from $$M^{p_{1}, q_1}_{w_1}\times M^{p_{2}, q_2}_{w_2} \times
\hdots \times M^{p_{m}, q_m}_{w_m}\quad \textrm{ into}\quad  M^{p_{0},
  q_{m+1}}_{w_0}.$$ Moreover, there exists a constant $C$ so that for all  $\bm f$, we have
\begin{align*}
 \|T_\sigma \bm f\|_{M_{w_0}^{p_0q_{m+1}}}\leq C \|\sigma\|_{\MV}
\|f_1\|_{M_{w_1}^{p_1,q_1}}\ldots \|f_m\|_{M_{w_m}^{p_m,q_m}} .
\end{align*}
\end{theorem}
 
\begin{proof}
%
%
%

Let  $f_k\in M_{w_k}^{p_k,q_k}$, $k=1,\ldots m$,  $\varphi  \in S(\Rd)$, and denote $\bm\varphi=(\varphi,\ldots,\varphi)$. Note that
   $$\sup\{ |\langle \cdot ,g \rangle|,\ \ g\in M_{1/ w_0}^{p_0',q_{m+1}'} \} $$ defines a norm which is equivalent to $\|\cdot\|_{M_{w_0}^{p_0q_{m+1}}}$ for $p_0,q_{m+1}\in [1,\infty]$ (see, for example,  \cite[Proposition 1.2(3)]{Tof07}). Hence, to complete our result
 on the basis of  Lemma~\ref{lemmaT_A}, we estimate for $ g\in M_{1/ w_0}^{p_0',q_{m+1}'}$ as follows
\begin{align*}
|\ip{T_\sigma \bm f}{g}| &=|\ip{\sigma}{\overline{R(\bm f, g)}}|=|\ip{V_{\overline{R(\bm \varphi, \varphi)}}\sigma}{V_{\overline{R(\bm \varphi, \varphi)}}\overline{R(\bm f, g)}}\notag \\
& \leq \|\sigma\|_{\MV} \|R(\bm f, g)\|_{\mathcal{M}_{1/w}^{(r_0',\bm r')_{\kappa}, (\bm s',s_{m+1}')_{\rho}}}.
\end{align*}
Using the conjugate indices $r_0', {\bf r'_{\kappa}}, s_{m+1}',  \bf s'_\rho$, it is easy to see that the conditions on the indices (1)--(4) are equivalent to those in Proposition~\ref{Vtilde}. Therefore, 
\[ \|R(\bm f, g)\|_{\mathcal{M}_{1/w}^{(r_0',\bm r')_{\kappa}, (\bm s',s_{m+1}')_{\rho}}} \leq C \|f_1\|_{M_{w_1}^{p_1,q_1}}\ldots \|f_m\|_{M_{w_m}^{p_m,q_m}}\|g\|_{M_{w_{0}}^{p_0', q_{m+1}'}}. \qedhere \]
\end{proof}

Note that the criteria on time and frequency are separated.  Even when it comes to order of integration, we do not link these, that is, the permutations $\kappa$ and $\rho$ are not necessarily identical.

\begin{corollary}\label{thm:increasing}
If 
\begin{eqnarray*}
  1&\leq r_0'\leq p_1 \leq r_1' \leq p_2 \leq \ldots \leq r_{m-1}' \leq p_{m}  \leq r_m' &\leq \infty;\\
                1&\leq s_{1}' \leq q_1 \leq s_2' \leq q_2 \leq \ldots \leq q_{m-1} \leq s_{m}' \leq q_m &\leq s_{m+1}' \leq \infty;
\end{eqnarray*}
and
\begin{align*}
 \frac 1 {r_0}+ \sum_{\ell =1}^m  \frac 1{ p_{\ell}}  + \frac 1 {  r_{\ell}} &\geq m + \frac 1 {p_0} ;\\
  \frac 1 {s_{m+1}}+ \sum_{\ell =1}^m \frac 1{q_{\ell}}  + \frac 1 { s_{ \ell}}   &\geq m+\frac 1{ q_{m+1}}  ;
\end{align*} 
then the conclusion of Theorem~\ref{thm:main} for any symbol  $\sigma
\in \mathcal{M}_v^{(r_0,\bm r)_{\kappa}, (\bm s,s_0)_{\rho}},$ where $\kappa, \rho$ are
the identity permutations. 
\end{corollary}

\begin{proof}
 Note that since $\kappa, \rho$ are the identity permutations, then $z=0$ and $\omega=m+1$. 
 
\begin{enumerate}
  \item $\displaystyle  \frac 1{ r_0} + \frac 1 { p_{k+1}}+\sum_{\ell =1}^{k}   \frac 1{ p_{\ell}} + \frac 1 {  r_{\ell}} \leq k+1,  \ k=0,\ldots, m{-}1$;
\item $\displaystyle \frac 1{  {s}_{m+1}}  +\frac 1 {q_k}+ \sum_{\ell =k+1}^{m} \frac 1{q_{\ell}}  + \frac 1 { s_{ \ell}} \geq  m-k +  1     , \quad k=1,\ldots, m$,
\end{enumerate}
follow from the monotonicity conditions.
\end{proof}

\section{Applications}\label{sec:applications}
 In Section~\ref{subsec4.1} we simplify the conditions  of Theorem~\ref{thm:main} in case of bilinear operators, that is, $m=2$. The focus of Section~\ref{sub4.2} lies on establishing  boundedness of the bilinear Hilbert transform on products of modulation
spaces. We stress  that these results are beyond the reach of  existing methods of time-frequency analysis of bilinear pseudodifferential operators as developed in \cite{BGHO, BeOk04, BeOk06, abekok07}. Finally, in
Section~\ref{sec5} we  consider the trilinear Hilbert transform.

\subsection{Bilinear pseudodifferential operators}\label{subsec4.1}

A bilinear
pseudodifferential operator with symbol $\sigma$ is formally defined by 
\begin{equation}\label{bpsido}
T_{\sigma}(f, g) (x) = \iint_{\Rr\times
  \Rr}\sigma (x,\xi_1,\xi_2)\hat{f}(\xi_1)\hat{g}(\xi_2)\, d\xi_1\,
d\xi_2.
\end{equation}
For $m=2$, Theorem~\ref{thm:main} simplifies to the following.

\begin{theorem}\label{thm:bilinear}
Let $1\leq p_0, p_1,p_2, q_1,q_2, q_3, r_0,r_1,r_2, s_1,s_2, s_3
\leq \infty$. If $$1/p_1+1/ r_1,\ 1/p_2+1/ r_2\geq 1$$
and one of the following
\begin{align} \tag{1}
 \frac 1 {p_0} &\leq \frac 1  {r_0}, 
 				\ \ \quad\quad \qquad\qquad\qquad\qquad  \quad (\text{using } \kappa=(1,2,0) \text{ or } (2,1,0));\\
 \tag{2}1+ \frac 1 {p_0} &\leq \frac 1{r_0} + \frac1{r_1} + \frac 1{ p_1 },\quad  r_1\leq p_0,r_0,   \quad  ( \kappa=(2,0,1)); \\  \tag{3}
 1+\frac 1 {p_0} &\leq \frac 1{r_0} + \frac1{r_2} + \frac 1{ p_2 },\quad  r_2\leq p_0,  r_0, \quad (  \kappa=(1,0,2)); \\  \tag{4}
2+ \frac 1 {p_0}&\leq \frac 1 {r_0} + \frac 1{r_1}+\frac 1{r_2} +\frac 1{\max\{p_1, r_0'\}}+\frac 1{p_2},\quad  r_2\leq p_0,\quad  r_1,r_2 \leq r_0, \quad (\kappa=(0,1,2));
 \\  \tag{5}
2+ \frac 1 {p_0}&\leq \frac 1 {r_0} + \frac 1{r_1}+\frac 1{r_2} +\frac 1{\max\{p_2, r_0'\}}+\frac 1{p_1},\quad  r_1\leq p_0,\quad  r_1,r_2 \leq r_0, \quad (\kappa=(0,2,1));
\end{align}
as well as one of
\begin{align}   
   \tag{1} \frac 1 {q_3} & \leq \frac 1 {s_3}, \ \ \qquad   (\text{using } \rho=(3,1,2) \text{ or } (3,2,1));\\  \tag{2}
1+\frac 1 {q_3}&\leq \frac 1 {q_1}+\frac 1 {s_1}+\frac 1 {s_3} , \quad s_3 \leq q_3,s_1,q_1' ,	
									\qquad 
				 ( \rho=(1,3,2));\\  \tag{3}
				1+\frac 1 {q_3}&\leq \frac 1 {q_2}+\frac 1 {s_2}+\frac 1 {s_3} , \quad s_3 \leq q_3,s_2,q_2' 	,
									\qquad 
				 ( \rho=(2,3,1));\\  \tag{4}
 2 &\leq\frac 1{\max\{q_1,s_1'\}}+\frac 1{\max\{q_2,s_2'\}}+ \frac 1 {s_2}+\frac 1 {s_3}, \quad s_3\leq q_2',s_2,\\
\nonumber 2+\frac 1 {q_3} &\leq \frac 1{\max\{q_1,s_1'\}}+\frac 1{\max\{q_2,s_2'\}} +\frac 1 {s_1}+\frac 1 {s_2}+\frac 1 {s_3},\quad ( \rho=(1,2,3));\\  \tag{5}
  2 &\leq\frac 1{\max\{q_1,s_1'\}}+\frac 1{\max\{q_2,s_2'\}}+ \frac 1 {s_1}+\frac 1 {s_3}, \quad s_3\leq q_1',s_1,\\
\nonumber 2+\frac 1 {q_3} &\leq \frac 1{\max\{q_1,s_1'\}}+\frac 1{\max\{q_2,s_2'\}} +\frac 1 {s_1}+\frac 1 {s_2}+\frac 1 {s_3},\quad ( \rho=(1,3,2)),
\end{align}
hold. Assume that $w_0, w_1, w_2,$ and $v$ are weight functions
satisfying 
\begin{equation*}\label{2dweight}
 v(x,t_1,t_2,\xi_1,\xi_2,\nu)^{-1}\leq  w_0 (x,\nu+\xi_1+\xi_2)^{-1}\cdot w_1(x-t_1,\xi_1) \cdot w_2(x-t_2,\xi_2).
\end{equation*}  
If $\sigma \in \mathcal{M}^{(r_0,r_1,r_2),{\kappa}; (s_1,s_2,s_3),{\rho}}$, the  bilinear
pseudodifferential operator $T_\sigma$ initially defined on $S(\Rd) \times
S(\Rd)$ by~\eqref{bpsido} extends to a bounded bilinear operator from
$M_{w_1}^{p_1,q_1} \times M_{w_2}^{p_2,q_2}$ into $M_{w_0}^{p_0,q_3}$.
Moreover, there exists a constant $C>0$, such that we have
\begin{align*}
 \|T_{\sigma}(f_1,f_2)\|_{M_{w_0}^{p_0,q_3}}\leq C \|\sigma\|_{\mathcal{M}^{(r_0,r_1,r_2),{\kappa}; (s_1,s_2,s_3),{\rho}}}\ 
\|f_1\|_{M_{w_1}^{p_1,q_1}}\ \|f_2\|_{M_{w_2}^{p_2,q_2}} 
\end{align*}
with appropriately chosen order of integration $\kappa,\rho$.
%
%
\end{theorem}

\begin{proof}
 This  result is derived from Theorem~\ref{thm:main}  for $m=2$, namely, we establish conditions on the $p_0,p_1,p_2,r_0, r_1,r_2,q_1,q_2,q_3,s_1,s_2,s_3$ for the existence of $\widetilde p_1,\widetilde p_2,\widetilde q_1,\widetilde q_2 $ satisfying the conditions of Theorem~\ref{thm:main}.
 
 If $\kappa=(1\ 2\ 0)$ or $\kappa=(2\ 1\ 0)$, then $z=2$ in Theorem~\ref{thm:main} and we require in addition only  $\frac 1 {r_0}\geq \frac 1 {p_0}$.
 
  For the remaining cases, we have to show that the conditions above imply the existence of $\widetilde p_1\geq p_1$,  $\widetilde p_2\geq p_2$ which allow for the application of Theorem~\ref{thm:main}.  
 
 If $\kappa=(1\ 0\ 2)$, we have $z=1$,  and  we seek, with notation as before, $\widetilde P_1$ and  $\widetilde P_2$ with
  \begin{align*}
 \widetilde P_1 &\leq P_1;\qquad  &\widetilde P_2 &\leq P_2;\\
 \widetilde P_1 +R_1&\geq 1;\qquad  &\widetilde P_2 +R_2&\geq 1;\\
  R_0+\widetilde P_2 &\leq 1;\\
 R_0+ \widetilde P_2 + R_2  &\geq 1+P_0;&
\end{align*}
that is,
\begin{align} \notag 
1-R_1&\leq \widetilde P_1 \leq P_1;\\ 
\label{1strip} 1-R_0- R_2 +P_0,\  1-R_2&\leq \widetilde P_2 \leq P_2,\ 1-R_0;
\end{align}
which defines a non empty set if and only if $P_1+R_1\geq 1$, $P_2+R_2\geq 1$, $R_2\geq P_0, R_0$, $1+P_0\leq R_0+R_1+P_2$.

For $\kappa=(0\ 1\ 2)$ we have  $z=0$ in Theorem~\ref{thm:main} and we require that some $\widetilde P_1$ and  $\widetilde P_2$ satisfy
\begin{align*}
 \widetilde P_1 &\leq P_1;\qquad  &\widetilde P_2 &\leq P_2;\\
 \widetilde P_1 +R_1&\geq 1;\qquad  &\widetilde P_2 +R_2&\geq 1;\\
  R_0+\widetilde P_2 &\leq 1;\qquad  &R_0+ \widetilde P_2 + \widetilde P_1 +R_1&\leq 2;\\
 R_0+ \widetilde P_2 + R_2 + \widetilde P_1 +R_1&\geq 2+P_0;&&
\end{align*}
that is,

\begin{align} \label{verticalstrip}
1-R_1&\leq \widetilde P_1 \leq P_1,\, 1-R_0;\\
\label{horizontalstrip}1-R_2&\leq \widetilde P_2 \leq P_2;\\
\label{diagonalstrip}2+P_0- R_0-R_1-R_2 &\leq \widetilde P_1 + \widetilde P_2 \leq 2- R_0-R_1.
\end{align}

Note that \eqref{verticalstrip} defines a vertical strip in the $(\widetilde P_1,\widetilde P_2 )$ plane which is non-empty if and only if $P_1+R_1\geq 1$ and $R_0\leq R_1$. Similarly, \eqref{horizontalstrip} defines a horizontal strip which is not empty if we assume $P_2+R_2\geq 1$.  Lastly, the diagonal strip given by \eqref{diagonalstrip} is nonempty if and only if $P_0\leq R_2$.  

To obtain a boundedness result, we still need to establish that the diagonal strip meats the rectangle given by the intersection of horizontal and vertical strips.  This is the case if the  upper right hand corner of the rectangle is above the lower diagonal given by $ \widetilde P_1 + \widetilde P_2=2+P_0- R_0-R_1-R_2$, that is, if 
\begin{align*}
 \min\{P_1,\, 1-R_0\}+P_2 \geq 2+P_0- R_0-R_1-R_2,
\end{align*}
and if the lower left corner of the rectangle lies below the upper diagonal, that is, if 
\begin{align*}
1-R_1+1-R_2 \leq 2- R_0-R_1,
\end{align*}
which holds if $R_0\leq R_2$.

Let us now turn to the frequency side. 
If $\rho=(3\ 2\ 1)$ or $\rho=(3\ 1\ 2)$, we have $w=1$ and  an application Theorem~\ref{thm:main} requires the single but strong assumption $Q_3\leq S_3$.

For $\rho=(1\ 3\ 2)$ we have  $w=2$ in Theorem~\ref{thm:main}. To satifiy the conditions, we need to establish the existence of   $\widetilde Q_1$ and  $\widetilde Q_2$ satisfy
\begin{align*}
 \widetilde Q_1 &\leq Q_1;\qquad  &\widetilde Q_2 &\leq Q_2;\\
 \widetilde Q_1 +S_1&\leq 1;\qquad  &\widetilde Q_2 +S_2&\leq 1;\\
  S_3+\widetilde Q_1 &\geq 1;\qquad  &S_3+ \widetilde Q_1 +S_1&\geq 1+Q_3.
  \end{align*}
The existence of such $\widetilde Q_2$ is trivial, so we are left with 
\begin{align*} 
  1+Q_3-S_1-S_3,\ 1-S_3 \leq \widetilde Q_1 &\leq Q_1,\ 1-S_1,.
\end{align*}
Note that this inequality is exactly \eqref{1strip} with $S_3$ replacing $R_2$, $S_1$ replacing $R_0$ $Q_3$ replacing $P_0$, and $Q_1$, $\widetilde Q_1$ in place of $P_2$, $\widetilde P_2$.  

We conclude that for the existence of $\widetilde Q_1$, we require $S_3\geq S_1,Q_3,1-Q_1$, and $1+Q_3\leq Q_1+S_1+S_3$.

For $\rho=(1\ 2\ 3)$ we have  $w=3$ in Theorem~\ref{thm:main}. We need to establish the existence of   $\widetilde Q_1$ and  $\widetilde Q_2$ satisfy
\begin{align*}
 \widetilde Q_1 &\leq Q_1;\qquad  &\widetilde Q_2 &\leq Q_2;\\
 \widetilde Q_1 +S_1&\leq 1;\qquad  &\widetilde Q_2 +S_2&\leq 1;\\
  S_3+\widetilde Q_2 &\geq 1;\qquad  &S_3+ \widetilde Q_2 + \widetilde Q_1 +S_2&\geq 2;\\
 S_3+ \widetilde Q_1 + \widetilde Q_2  +S_2 +S_1&\geq 2+Q_3;&&
\end{align*}
that is, choosing 
\begin{align*}
  \widetilde Q_1 =\min\{ Q_1,\, 1-S_1\},\quad \text{and}\quad  \widetilde Q_2 =\min\{ Q_2,\, 1-S_2\},
  \end{align*}
   we require 
\begin{align*} 1 &\leq \min\{ Q_2,\, 1-S_2\}+S_3;\\
 2 &\leq \min\{ Q_1,\, 1-S_1\}+\min\{ Q_2,\, 1-S_2\}+ S_2+S_3; \\
 2+Q_3&\leq \min\{ Q_1,\, 1-S_1\}+\min\{ Q_2,\, 1-S_2\} +S_1+S_2+S_3.\qedhere
\end{align*}
\end{proof}
 
\begin{proof}{\bf Proof of Theoremm~\ref{model-result}} 
  Theorem~\ref{model-result}  now follows from choosing $\kappa$ and $\rho$ to be the identity permutations, and $r_0=s_1=s_2=\infty$, $r_1=r_2=s_3=1$.
  \end{proof} Note that  this result covers  and extends Theorem 3.1 in \cite{BGHO}.
\begin{remark} \rm
Using Remark~\ref{remark:Minkowski}, we observe that $M^{\infty,1}(\RR^{3d}) \subsetneq \mathcal{M}^{(\infty, 1,1), (\infty, \infty, 1)}(\RR^{3d})$. Indeed, in both cases we have the same decay parameters, but different integration orders, namely
\begin{align*}
\begin{array}{lllllll}
 M^{\infty,1} &    x \to \infty, & \xi_1 \to \infty, & \xi_2 \to \infty,&  \nu \to 1, & t_1 \to 1,  &  t_2 \to 1 ;    \\
 \mathcal{M}^{(\infty, 1,1), (\infty, \infty, 1) }&  x \to \infty,  & t_1 \to 1,  &  t_2 \to 1,& \xi_1 \to \infty, & \xi_2 \to \infty, &  \nu \to 1.
\end{array}
\end{align*}
Inclusion follows from the fact that we always moved a large exponent to the right of a small exponent.  Note that for any $r\in M^{1,\infty}(\mathbb R)\setminus M^{\infty,1}(\mathbb R)$, for example, a chirped signal $r(\xi)=e^{2\pi i \xi^2} u(\xi)$ with $u(\xi)\in L^2 \setminus L^1$, we have $$\sigma(x,\xi_1,\xi_2)  =r(\xi_1)\in \mathcal{M}^{(\infty, 1,1), (\infty, \infty, 1) }\setminus M^{\infty,1}.$$ 
\end{remark}

\begin{example} \rm 
With $\kappa$ and $\rho$ are the identity, that is, $\kappa=(0,1,2)$ and $\rho=(1,2,3)$, we illustrate the applicability of Theorem~\ref{thm:bilinear} for maps on $L^2\times L^2=M^{2,2}\times M^{2,2}$, that is,  $p_1=p_2=q_1=q_2=2$. 

On the time side, we require $r_1, r_2\leq 2,r_0$ and $r_2\leq p_0$ as well as
$$
\frac 3 2+ \frac 1 {p_0}\leq \frac 1 {r_0} + \frac 1{r_1}+\frac 1{r_2} +\frac 1{\max\{2, r_0'\}}.
$$
Our goal is to obtain results for $r_0$ large, hence, we assume $r_0\geq 2$. (In case of $r_0\leq2$,  the last inequality above does not depend on $r_0$, and we can improve the result by fixing $r_0=2$.)
We obtain the range of applicability $r_1, r_2\leq 2\leq r_0$, and $r_2\leq p_0$, and 
$$
1+ \frac 1 {p_0}\leq \frac 1 {r_0} + \frac 1{r_1}+\frac 1{r_2}.
$$

On the frequency side, we have to satisfy the  conditions $s_3\leq 2, s_2$, 
\begin{align*}
 2 &\leq\frac 1{\max\{2,s_1'\}}+\frac 1{\max\{2,s_2'\}}+ \frac 1 {s_2}+\frac 1 {s_3},\\
 2+\frac 1 {q_3} &\leq \frac 1{\max\{2,s_1'\}}+\frac 1{\max\{2,s_2'\}} +\frac 1 {s_1}+\frac 1 {s_2}+\frac 1 {s_3}.
\end{align*}
Let us assume $s_1\leq 2\leq s_2$, then we have the range of applicability $s_1,s_3\leq 2\leq s_2$, 
\begin{align*}
 \frac 1 2 + \frac 1{s_1},  \frac 1 2 +\frac 1 {q_3} &\leq \frac 1 {s_2}+\frac 1 {s_3}.
\end{align*}

The range of applicability gives exponents that guarantee that a bilinear pseudodifferential operator maps boundedly $L^2 \times L^2$ into $M^{p_0, q_3}$ if 
 $\sigma \in \mathcal{M}^{(r_0,r_1,r_2), (s_1,s_2,s_3)}$.
 
In particular, when $\sigma \in \mathcal{M}^{(\infty, 1,1), (2,2, 1)}$ we can take  $p_0=q_3=1$.  So we get that $T_{\sigma}$ maps $L^2\times L^2$ into $M^{1,1}\subset M^{1, \infty}$. 
\end{example}

 \subsection{The bilinear Hilbert transform}\label{sub4.2}
 We now consider boundedness properties of the bilinear Hilbert transform on  modulation spaces. Recall that this operator is defined  for $f, g \in \S(\Rr)$ by 
 
 \begin{equation*}
 \bh(f, g)(x)=\lim_{\epsilon \to 0}\int_{|y|>\epsilon}f(x+y)g(x-y)\frac{dy}{y}. 
 \end{equation*}
 Equivalently, this operator can be written as a Fourier multiplier, that is, 
 a bilinear pseudodifferential operator whose symbol is
 independent of the space variable, with symbol $\sigma_{\bh}(x,\xi_1,\xi_2)=\sigma(\xi_1-\xi_2)$, where 
 $\sigma(x)=-\pi i {\rm sign}\,(x), x\neq 0$. 
 
 
 Our first goal is to identify which of the (unweighted) spaces $\mathcal{M}^{(r_0,r_1,r_2),{\kappa}; (s_1,s_2,s_0),{\rho}}$ the symbol $\sigma_\bh$ belongs to.  To this end consider the window function  $\Psi(x,\xi_1,\xi_2)=\psi(x)\psi(\xi_2)\psi(\xi_1-\xi_2)$, where $\psi \in \S(\Rr)$ such that $\psi(x)=\psi_1(x)-\psi_1(-x)$ with $\psi_1 \in \S(\Rr)$, $0\leq \psi_1(x)\leq 1$ for all $x\in \Rr$. In addition, we require that the support of $\psi_1$ is strictly included in $(0, 1).$ 
Then
\begin{align*}
 \mathcal V_\Psi\sigma_\bh (x,t_1,t_2,\xi_1,\xi_2,\nu)&=
 	V_\psi 1 (x,\nu) \ V_\psi\sigma (\xi_1-\xi_2,t_1) \  V_\psi 1 (\xi_2,t_1+t_2) 
\end{align*}

Assume that the two permutations $\kappa$ of  $\{0,1,2\}$, and $\rho$ of $\{1,2,3\}$ are
identities. Moreover, suppose that all the weights are identically
equal to $1$.
\begin{proposition}\label{symb-bht}
For  $r>1$, we have that 
%
$\sigma_\bh \in \mathcal{M}^{(\infty, 1, r), (\infty,  \infty, 1)}$.  \end{proposition}
 
\begin{proof}
Let $r>1$. We shall integrate $$\mathcal V_{\Psi}\sigma_\bh(x, \bm t, \bm \xi, \nu)= 
 	V_\psi 1 (x,\nu) \ V_\psi\sigma (\xi_1-\xi_2,t_1) \  V_\psi 1 (\xi_2,t_1+t_2) $$ in the  order
\begin{align*}
\begin{array}{llllll}
  x \to  r_0=\infty
 & t_1 \to  r_1=1
 & t_2 \to   r_2=r>1
 & \xi_1 \to   s_1= \infty
 & \xi_2 \to   s_2=\infty
 & \nu \to   s_0=1. 
\end{array}
\end{align*}
We  estimate 
\begin{align*}
\|\sigma_\bh\|_{\mathcal{M}^{(\infty, \infty, 1), (\infty, 1, r)}} &= \int_{\Rr}\sup_{\xi_1, \xi_2}\bigg(\int_{\Rr}\bigg(\int_{\Rr}\sup_{x}|\mathcal V_{\Psi}\sigma_\bh(x, \bm t, \bm \xi, \nu)|dt_1\bigg)^{r} dt_2\bigg)^{1/r}d\nu\\
&= \int_{\Rr}\sup_{\xi_1, \xi_2}\bigg(\int_{\Rr}\bigg(\int_{\Rr}\sup_{x} | V_\psi 1 (x,\nu) \ V_\psi\sigma (\xi_1-\xi_2,t_1) \  V_\psi 1 (\xi_2,t_1+t_2) |dt_1\bigg)^{r} dt_2\bigg)^{1/r}d\nu       \\
&= \|\hat{\psi}\|_{L^1} \sup_{\xi_1, \xi_2}\bigg(\int_{\Rr}\bigg(\int_{\Rr} |V_\psi\sigma (\xi_1-\xi_2,t_1) \  V_\psi 1 (\xi_2,t_1+t_2) |dt_1\bigg)^{r} dt_2\bigg)^{1/r}\\
& \leq \|\hat{\psi}\|_{L^1} \sup_{\xi_1, \xi_2} \| |V_\psi\sigma(\xi_1-\xi_2, \cdot)|\ast|V_\psi1(\xi_2, \cdot)|\|_{L^{r}}\\
&\leq  \|\hat{\psi}\|_{L^1} \sup_{\xi_1, \xi_2} \|V_\psi\sigma(\xi_1-\xi_2, \cdot)\|_{L^{r}}\|V_\psi1(\xi_2, \cdot)\|_{L^1}\\
&=  \|\hat{\psi}\|_{L^1}^2 \sup_{\xi_1} \|V_\psi\sigma(\xi_1 \cdot)\|_{L^{r}},
\end{align*}
where we have repeatedly used the fact that  $V_{\psi}1(x,\nu)=e^{2\pi i x \nu}\hat{\psi}(\nu)$, and
$V_{\psi}1\in L^{\infty}(x)L^{1}(\nu),$ that is
$$\int_{\Rr}\sup_{x}|V_{\psi}1(x, \nu)|d\nu=\|\hat{\psi}\|_{L^1}<\infty.$$ 
Thus, we are left
to estimate $$\sup_{\xi} \|V_\psi\sigma(\xi \cdot)\|_{L^{r}}.$$


%
 
Recall that $\psi(x)=\psi_1(x)-\psi_1(-x)$, hence, we have 
 $$V_{\psi}\sigma(\xi, t)=e^{-2\pi i \xi
   t}\bigg[-\int_{-\infty}^{-\xi}e^{-2\pi iy t}\psi(y)dy+
 \int_{-\xi}^{\infty}e^{-2\pi it y}\psi(y)dy\bigg].$$
A series of straightforward calculations yields 
 \begin{align*}
  |V_{\psi} \sigma (\xi,t)|= 
\begin{cases}
|\hat{\psi_1}(t) - \hat{\psi_1}(-t)| &\mbox{if } |\xi|\geq 1 \\
	|\hat{\psi_1}(-t) - \widehat{\chi_{[0, -\xi]}}\ast \hat{\psi_1}(t) + \widehat{\chi_{[-\xi, 1]}}\ast \hat{\psi_1}(t)|&\mbox{if } -1\leq \xi \leq 0 \\
	|\hat{\psi_1}(t) - \widehat{\chi_{[\xi, 1]}}\ast \hat{\psi_1}(-t) + \widehat{\chi_{[0, \xi]}}\ast \hat{\psi_1}(-t)|&\mbox{if } 0\leq \xi \leq 1, 
\end{cases}
\end{align*}
where $\chi_{[a, b]}$ denotes the characteristic function of $[a, b]$. 
We note that  that $\widehat{\chi_{[0, -\xi]}}, \widehat{\chi_{[-\xi, 1]}}, \widehat{\chi_{[\xi, 1]}}, \widehat{\chi_{[0, \xi]}} \in L^r$ uniformly for $|\xi|\leq 1$ for each $r>1$. 

For $|\xi|\geq 1$, we have $$\|V_\psi\sigma(\xi, \cdot)\|_{L^q}\leq 2\|\hat{\psi_{1}}\|_{L^q}$$ for any $q\geq 1$. 
Now consider $-1\leq \xi \leq 0$, then 
\begin{align*}
\|V_\psi\sigma(\xi, \cdot)\|_{L^r}& \leq  \|\hat{\psi_1}\|_{L^r}+ \|\widehat{\chi_{[0, -\xi]}}\ast \hat{\psi_1}\|_{L^r} + \|\widehat{\chi_{[-\xi, 1]}}\ast \hat{\psi_1}\|_{L^r}\\
&\leq \|\hat{\psi_1}\|_{L^r} + \|\hat{\psi_1}\|_{L^1}(\|\widehat{\chi_{[0, -\xi]}}\|_{L^r} + \|\widehat{\chi_{[-\xi, 1]}}\|_{L^r})\\
& \leq  \|\hat{\psi_1}\|_{L^r} + C \|\hat{\psi_1}\|_{L^1}
\end{align*}
where $C>0$ is a constant that depends only on $r$. Using a similar estimate for $0\leq \xi \leq 1$, we conclude that $$\sup_{\xi}\|V_\psi\sigma(\xi, \cdot)\|_{L^r} \leq C< \infty$$ where $C$ depends only on $\psi_1$ and $r$. 
%
%
\end{proof}
 
Observe that $\sigma_\bh \in \mathcal{M}^{(\infty, 1, r), (\infty,  \infty, 1)} (\RR^3)\setminus \mathcal{M}^{(\infty, 1, 1), (\infty,  \infty, 1)}(\RR^3)$ for all $r>1$. Consequently, to obtain a boundedness result for the bilinear Hilbert transform, we cannot apply any of the existing results on bilinear pseudodifferential operators. However, using the  symbol classes  introduced we obtain the following result: 

\begin{theorem}\label{bdbht}
Let  $1\leq  p_0, p_1, p_2, q_1, q_2,q_3\leq \infty $ satisfy $\frac{1}{p_1}+\frac{1}{p_2} > \frac{1}{p_0}$  and that $\tfrac{1}{q_1}+\tfrac{1}{q_2}\geq 1+\tfrac{1}{q_3}$. Then the bilinear Hilbert transform
  extends to a bounded bilinear operator from $M^{p_1, q_1} \times
  M^{p_2, q_2}$ into $M^{p_0,q_3}$. Moreover, there exists a constant $C>0$
  such that 
$$ \|\bh(f, g)\|_{M^{p_0, q_3}}\leq C \|f\|_{M^{p_1,q_1}} \|g\|_{M^{p_2,q_2}}.$$ 

In particular, for any $1\leq p, q \leq \infty,$ and  $\epsilon>0$, the $\bh$ continuously maps $M^{p,q}\times M^{p', q'}$ into $M^{1+\epsilon,\infty}$ and we have 
$$ \|\bh(f, g)\|_{M^{1+\epsilon, \infty}}\leq C   \|f\|_{M^{p,q}} \|g\|_{M^{p',q'}}.$$
\end{theorem} 

%
%
%
%

\begin{proof}
Since the symbol $\sigma_\bh$ of $\bh$ satisfies $\sigma_\bh \in \mathcal{M}^{(\infty, 1, r), (\infty,
   \infty, 1)}$, the proof follows from Theorem~\ref{thm:bilinear}. Indeed, on the time side, all simple inequalities hold and we are left to check 
 $$  2+ \frac 1 {p_0}\leq \frac 1 {r_0} + \frac 1{r_1}+\frac 1{r_2} +\frac 1{\max\{p_1, r_0'\}}+\frac 1{p_2},$$
 which is with $\frac 1 r=1-\epsilon$
 $$  2+ \frac 1 {p_0} \leq 0+1+1-\epsilon +\frac 1{\max\{p_1, 1\}}+\frac 1{p_2}.
 $$
 On the frequency side, the conditions
\begin{align*}
  2 &\leq\frac 1{\max\{q_1,s_1'\}}+\frac 1{\max\{q_2,s_2'\}}+ \frac 1 {s_2}+\frac 1 {s_3}, \quad s_3\leq q_2',s_2,\\
 2+\frac 1 {q_3} &\leq \frac 1{\max\{q_1,s_1'\}}+\frac 1{\max\{q_2,s_2'\}} +\frac 1 {s_1}+\frac 1 {s_2}+\frac 1 {s_3},
\end{align*} are clearly satisfied whenever
\[
	 2+\frac 1 {q_3} \leq \frac 1{\max\{q_1,1\}}+\frac 1{\max\{q_2,1\}} +0+0+0.\qedhere
\]
\end{proof}

\begin{remark}\rm
 It was proved in \cite{mlct97, mlct99} that the bilinear Hilbert transform $\bh$ continuously maps $ L^{p_1} \times L^{p_2}$ into $L^{p}$ where $\tfrac{1}{p}=\tfrac{1}{p_1}+\tfrac{1}{p_2}$, $1\leq p_1, p_2 \leq \infty$ and $2/3 < p \leq \infty$.  
Our results give that if  $1<p, q, p_1 < \infty$ then $H$ maps continuously $M^{p_1,q}\times M^{p_1', q'}$ into $M^{p, \infty}$.  

One can use embeddings between modulation spaces and Lebesgue spaces to get some ``mixed'' boundedness results. For example, assume that $q\geq 2$ and  $q'\leq p_1 \leq q$, then it is known that (see \cite[Proposition 1.7]{Toft1})  $$L^{p_1} \subset M^{p_1,q}\quad {\textrm and}\quad M^{p_{1}', q'}\subset L^{p_{1}'}.$$  Consequently, it follows from Theorem~\ref{bdbht} that $\bh$ continuously maps $L^{p_1}\times M^{p'_{1}, q'} $ into $M^{p, \infty}\supset L^p$.  
%
%
%
\end{remark}

\subsection{The trilinear Hilbert transform}\label{sec5}

In this final section we consider  the trilinear Hilbert transform $\th$ given formally by 
\begin{equation*}
\th(f, g, h)(x)=\lim_{\epsilon\to 0}\int_{|t|>\epsilon}f(x-t)g(x+t)h(x+2t)\frac{dt}{t}.
\end{equation*}

The trilinear Hilbert transform can be written as a trilinear pseudodifferential operator, or more specifically as a trilinear Fourier multiplier given by  $$\th(f, g, h)(x)=\iiint_{\Rr\times \Rr \times \Rr}\sigma_{\th }(x, \xi_1, \xi_2, \xi_3)\hat{f}(\xi_1)\hat{g}(\xi_2)\hat{h}(\xi_3)e^{2\pi i x(\xi_1+\xi_2+\xi_3)}d\xi_1d\xi_2d\xi_3$$ where $$\sigma_{\th}(x, \xi_1, \xi_2, \xi_3)= \sigma(\xi_1-\xi_2-2\xi_3)=\pi i \text{sign}(\xi_1-\xi_2-2\xi_3).$$

Recall from Section~\ref{sub4.2} that  $\psi \in \S(\Rr)$ is chosen such that $\psi(x)=\psi_1(x)-\psi_1(-x)$ with $\psi_1 \in \S(\Rr)$, $0\leq \psi_1(x)\leq 1$ for all $x\in \Rr$.
Next we define  $\Psi(x, \xi_1, \xi_2, \xi_3)=\psi(x)\psi(\xi_2)\psi(\xi_3)\psi(\xi_1-\xi_2-2\xi_3). $ 
We can now compute the symbol window Fourier transform
 $\mathcal{V}_{\Psi}\sigma_{\th}$ of   $\sigma_{\th}$ with respect to $\Psi$ and obtain 
$$
\mathcal{V}_{\Psi}\sigma_{\th}(x, \bm t, \bm \xi, \nu)=V_{\psi}1(x, \nu) V_{\psi}1(\xi_2, -t_1-t_2)  V_{\psi}1(\xi_2, -2t_1 - t_3) V_{\psi}\sigma(\xi_1-\xi_2-2\xi_3, -t_1)|.$$ Observe that $|V_g1(x,\eta)|=|\hat{g}(\eta)|$. Hence,
$$|\mathcal{V}_{\Psi}\sigma_{\th}(x, \bm t, \bm \xi, \nu)|=|\hat{\psi}( \nu)| |\hat{\psi}(-t_1-t_2)| |\hat{\psi}( -2t_1 - t_3)| | V_{\psi}\sigma(\xi_1-\xi_2-2\xi_3, -t_1)|.$$ But by the choice of $\psi$ we see that $\hat{\psi}(-\eta)=-\hat{\psi}(\eta)$. 

\begin{proposition}\label{symb-tht}
For $r>1$, we have 
$\sigma_{\th} \in \mathcal{M}^{(\infty, 1, r, r), (\infty,  \infty, \infty,  1)}$. In particular, this  conclusion  holds when $r=1+\epsilon$ for all $\epsilon>0.$
 \end{proposition}
 
\begin{proof}
Let $r>1$. We proceed as in the proof of Proposition~\ref{symb-bht}, and  integrate 
$$\mathcal V_{\Psi}\sigma_{\th }(x, \bm t, \bm \xi, \nu)$$ in the following order: 
\begin{align*}
\begin{array}{llll}
 x \to  r_0=\infty,
&t_1 \to   r_1=1,
&t_2 \to   r_2=r>1,
&t_3 \to   r_3=r>1,\\
 \xi_1 \to  s_1= \infty,
&\xi_2 \to   s_2=\infty,
&\xi_3 \to   s_3=\infty,
&\nu \to   s_0=1. 
\end{array}
\end{align*}
In particular, we  estimate 
\begin{align*}
\|\sigma_{\th}\|_{\mathcal{M}^{(\infty, 1, r, r), (\infty, \infty, \infty, 1)}} &=\int_{\RR}d\nu \sup_{\xi_1, \xi_2, \xi_3} \bigg(\int_{\RR}dt_{3}\int_{\RR}dt_2\int_{\RR}dt_1 \sup_{x}|\mathcal{V}_{\Psi}\sigma_{H}(x, \bm t, \bm \xi, \nu)|^r\bigg)^{1/r}\\
& = \int_{\RR}d\nu \sup_{\xi_1, \xi_2, \xi_3} \bigg(\int_{\RR}dt_{3}\int_{\RR}dt_2\int_{\RR}dt_1 \sup_{x} |\hat{\psi}( \nu)|^r |\hat{\psi}(-t_1-t_2)|^r \\
& \qquad |\hat{\psi}( -2t_1 - t_3)|^r | V_{\psi}\sigma(\xi_1-\xi_2-2\xi_3, -t_1)|^r\bigg)^{1/r}\\
&= \|\hat{\psi}\|_1 \sup_{\xi_1, \xi_2, \xi_3} \bigg(\int_{\RR}dt_{3}\int_{\RR}dt_2\int_{\RR}dt_1  |\hat{\psi}(-t_1-t_2)|^r  |\hat{\psi}( -2t_1 - t_3)|^r \\
& \qquad | V_{\psi}\sigma(\xi_1-\xi_2-2\xi_3, -t_1)|^r\bigg)^{1/r}\\
&= \|\hat{\psi}\|_1 \sup_{\xi_1, \xi_2, \xi_3} \bigg(\int_{\RR}dt_{3}\int_{\RR}dt_2\int_{\RR}dt_1  |\hat{\psi}(t_2+t_1)|^r  |\hat{\psi}(t_3 + 2t_1)|^r \\
& \qquad | V_{\psi}\sigma(\xi_1-\xi_2-2\xi_3, -t_1)|^r\bigg)^{1/r}\\
&= \|\hat{\psi}\|_1 \sup_{\xi_1, \xi_2, \xi_3} \bigg(\int_{\RR}dt_{3}\int_{\RR}dt_2\int_{\RR}dt_1  |\hat{\psi}(t_2-t_1)|^r  |\hat{\psi}(2(t_3 - t_1))|^r \\
& \qquad | V_{\psi}\sigma(\xi_1-\xi_2-2\xi_3, t_1)|^r\bigg)^{1/r}\\
&= \|\hat{\psi}\|_1 \sup_{\xi_1, \xi_2, \xi_3} \bigg(\int_{\RR}dt_{3}\int_{\RR}dt_2\int_{\RR}dt_1  |\hat{\psi}(t_1)|^r  |\hat{\psi}(2(t_2 -t_3 - t_1))|^r \\
& \qquad | V_{\psi}\sigma(\xi_1-\xi_2-2\xi_3, t_2-t_1)|^r\bigg)^{1/r}\\
&= \|\hat{\psi}\|_1 \sup_{\xi_1, \xi_2, \xi_3} \bigg(\int_{\RR}dt_{3}\int_{\RR}dt_2 |\hat{\psi}|^r \ast ( |T_{t_3}\hat{\psi}_{2}|^r  |\widetilde{V_{\psi}\sigma}(\xi_1-\xi_2-2\xi_3, \cdot)|^r)(t_2)\bigg)^{1/r}\\
\end{align*}
where $\hat{\psi}_{2}(\xi)=\hat{\psi}(2\xi)$, and $\widetilde{V_{\psi}\sigma}(\xi_1-\xi_2-2\xi_3, \eta)=V_{\psi}\sigma(\xi_1-\xi_2-2\xi_3, -\eta).$
Consequently,

\begin{align*}
\|\sigma_{\th}\|_{\mathcal{M}^{(\infty, 1, r, r), (\infty, \infty, \infty, 1)}} & \leq  \|\hat{\psi}\|_1  \|\hat{\psi}\|_r \sup_{\xi_1, \xi_2, \xi_3} \bigg(\int_{\RR}dt_{3}\int_{\RR}dt_2   |\hat{\psi}_{2}(t_2-t_3)|^r  |\widetilde{V_{\psi}\sigma}(\xi_1-\xi_2-2\xi_3, t_2)|^r)\bigg)^{1/r}\\
&= \|\hat{\psi}\|_1  \|\hat{\psi}\|_r \sup_{\xi_1, \xi_2, \xi_3} \bigg(\int_{\RR}dt_{3}   |\hat{\psi}_{2}|^r \ast  |\widetilde{V_{\psi}\sigma}(\xi_1-\xi_2-2\xi_3, \cdot)|^r (t_3)\bigg)^{1/r}\\
& \leq  \|\hat{\psi}\|_1  \|\hat{\psi}\|_r \|\hat{\psi}_2\|_{r} \sup_{\xi_1, \xi_2, \xi_3} \bigg(\int_{\RR}dt_{3}  |\widetilde{V_{\psi}\sigma}(\xi_1-\xi_2-2\xi_3, t_3)|^r \bigg)^{1/r}\\
& = \|\hat{\psi}\|_1  \|\hat{\psi}\|_r \|\hat{\psi}_2\|_{r} \sup_{\xi_1, \xi_2, \xi_3} \bigg(\int_{\RR}dt_{3}  |V_{\psi}\sigma(\xi_1-\xi_2-2\xi_3, t_3)|^r \bigg)^{1/r}.
\end{align*}
The proof is complete by observing that  the proof of Proposition~\ref{symb-bht} implies that 
\[\sup_{\xi_1, \xi_2, \xi_3} \bigg(\int_{\RR}dt_{3}  |V_{\psi}\sigma(\xi_1-\xi_2-2\xi_3, t_3)|^r \bigg)^{1/r} < \infty.\qedhere \]%
%
%
%
%
\end{proof}
Using this result and Theorem~\ref{thm:main} for $m=3$ we can give the following initial result on the boundedness of $\th$ on product of modulation spaces.


%
%
%

\begin{theorem}\label{boundeness-tht}
For $p,p_0,p_1\in (1,\infty)$ and $1\leq q \le \infty$,  the trilinear Hilbert transform $\th$   is bounded from $M^{p_1,1}\times M^{p, q}\times M^{p', q'}$ into $M^{p_0, \infty}$ and we have the following estimate:
$$\|\th(f, g, h)\|_{M^{p_0, \infty}}\leq C \|f\|_{M^{p_1,1}}  \|g\|_{M^{p, q}}  \|h\|_{M^{p', q'}}$$ for all $f, g, h \in S(\RR),$ where the constant $C>0$ is independent of $f, g, h$. 


\end{theorem}

\begin{remark}
Before proving this result we point out that the strongest results are  obtained by choosing $p_0$ as close to $1$ as possible and $p_1$ as close to $\infty$ as possible.

As special case, we see that $\th$  boundedly maps
$$M^{r,1}\times L^2 \times L^2\longrightarrow M^{1+\epsilon,\infty}$$
for every $r<\infty$ and $\epsilon>0$. 
\end{remark}

\begin{proof}We set $r=\min\{p_0,p_1',p,p'\}>1$.  The symbol of $\th$  is in the symbol modulation space with decay parameters $r_0=\infty, r_1=1, r_2=r_3=r>1$ as used in  Theorem~\ref{thm:main}. Note that here,  $\kappa$ is the identity permutation, so  $z=0$. The  boundedness conditions in  Theorem~\ref{thm:main} now read
\begin{itemize}
 \item[$k=0$\,:] \quad $\displaystyle 0+ \frac 1 {\widetilde p_1} \leq 1$;
  \item[$k=1$\,:] \quad $\displaystyle 0+ \frac 1 {\widetilde p_2} + \frac 1 {\widetilde p_1} +1 \leq 2$;
    \item[$k=2$\,:] \quad $\displaystyle 0+ \frac 1 {\widetilde p_3} + \frac 1 {\widetilde p_1} +1+\frac 1 {\widetilde p_2} +\frac 1 r  \leq 3$;
     \item[$k=3$\,:] \quad $\displaystyle 0 + \frac 1 {\widetilde p_1} +1+\frac 1 {\widetilde p_2} +\frac 1 r +\frac 1 {\widetilde p_3} +\frac 1 r  \geq 3+\frac 1 {p_0}$;
\end{itemize}
where
\begin{align}\label{eqn:threeSimple}
 p_1\leq \widetilde p_1\leq r_1'=\infty, \quad p_2\leq \widetilde p_2\leq r', \quad p_3 \leq \widetilde p_3\leq r'.
\end{align}
 The four conditions above reduce to
\begin{itemize}
  \item[$k=1$\,:] \quad $\displaystyle  \frac 1 {\widetilde p_1} + \frac 1 {\widetilde p_2} \leq 1$;
    \item[$k=2$\,:] \quad $\frac 1 {\widetilde p_1} +\frac 1 {\widetilde p_2} + \frac 1 {\widetilde p_3}   \leq 2-\frac 1 {r} $;
     \item[$k=3$\,:] \quad $\frac 1 {\widetilde p_1} +\frac 1 {\widetilde p_2} +\frac 1 {\widetilde p_3}  \geq 2-\frac 2 r +\frac 1 {p_0}$;
\end{itemize}
For simplicity, we now set $p_2=\widetilde p_2=p_3'=\widetilde p_3'\in [r,r']$ and obtain
\begin{itemize}
  \item[$k=1$\,:] \quad $\displaystyle  \frac 1 {\widetilde p_1}  \leq  \frac 1 {\widetilde p_3}$;
    \item[$k=2$\,:] \quad $\frac 1 {\widetilde p_1}   \leq 1-\frac 1 {r} $;
     \item[$k=3$\,:] \quad $\frac 1 {\widetilde p_1}  \geq 1-\frac 2 r +\frac 1 {p_0}$;
\end{itemize}
that is
\begin{itemize}
  \item[$k=1$\,:] \quad $\displaystyle  {\widetilde p_1}  \geq  {\widetilde p_3}$;
    \item[$k=2$\,:] \quad $ {\widetilde p_1}   \geq r' $;
     \item[$k=3$\,:] \quad $\frac 1 {\widetilde p_1}  \geq 1-\frac 2 r +\frac 1 {p_0}$;
\end{itemize}

Note that the condition for $k=1$ follows from the $k=2$ condition since $p_3'\leq r'$.

For the existence of $\widetilde p_1 \geq p_1$, satisfying the $k=2$ and $k=3$ conditions, we require 
$
2-\frac 1 {r}\geq 2-\frac 2 r +\frac 1 {p_0}
$,
which is $r\leq p_0$, a condition that is met. Some $\widetilde p_1\geq p_1$ will satisfy all conditions if $\frac 1 { p_1}  \geq 1-\frac 2 r +\frac 1 {p_0}$. Indeed,
$$
1-\frac 2 r +\frac 1 {p_0}= 1-\frac 1 r + \frac 1 {p_0} - \frac 1 r\leq  1-\frac 1 r \leq \frac 1 {p_1}.
$$

We now consider the conditions of Theorem~\ref{thm:main} on the frequency side. We choose $\rho$ to be the identity permutation on $\{1,2,3,4\}$, $s_1=s_2=s_3=\infty, s_4=1$.
We now have to consider existence of $\widetilde q_1\geq s_1'=1$, $\widetilde q_2\geq s_2'=1$, and $\widetilde q_3\geq s_3'=1$ with
\begin{itemize}
 \item[$k=1$\,:] \quad $\displaystyle \frac 1 {\widetilde q_1} +\frac 1 {\widetilde q_2} + \frac 1 {\widetilde q_3}\geq 2$;
  \item[$k=2$\,:] \quad $\displaystyle \frac 1 {\widetilde q_2} +   \frac 1 {\widetilde q_3} \geq 1$;
    \item[$k=3$\,:] \quad $\displaystyle   \frac 1 {\widetilde q_3} \geq 0 $ ;
     \item[$k=4$\,:] \quad $\displaystyle \frac 1 {\widetilde q_1} +\frac 1 {\widetilde q_2} +    \frac 1 {\widetilde q_3}\geq 2+\frac 1 {q_4}$.
\end{itemize}
These conditions reduce to
\begin{align*}
  \frac 1 {\widetilde q_2} +   \frac 1 {\widetilde q_3} \geq 1, \quad \frac 1 {\widetilde q_1} +\frac 1 {\widetilde q_2} +    \frac 1 {\widetilde q_3}\geq 2+\frac 1 {q_4}.
\end{align*}
To assume optimally large   $q_1,q_2,q_3$,  we choose $\tilde{q}_2=q_2=q$, $q_3'=\tilde{q}_3'=q'$ and $\frac 1 { q_1} = 1+\frac 1 {q_4}$, the latter only being satisfied if $q_1=1$ and $q_4=\infty$.
%
%
\end{proof}

In \cite[Theorem 13]{mtt} it is proved that the trilinear Hilbert transform is bounded from $L^p \times L^q \times \mathcal{A}$ into $L^r$ whenever $1<p, q\leq \infty$, $2/3< r < \infty$ and $\tfrac{1}{p}+\tfrac{1}{q}=\tfrac{1}{r},$ where $\mathcal{A}$ is the Fourier algebra. In particular, for $p=q=2$, then $r=1$ and the operator maps boundedly $L^2  \times L^2 \times \mathcal{A}$ into $L^1$. 

From \cite[Proposition 1.7]{Toft1} we know that when $p\in (1,2)$ and $p<q'<p'$, then $\mathcal{F}L^{q'} \subset M^{p',q'}$. We can then conclude that $\th$ continuously maps $M^{p_1,1}\times M^{p,q}\times \mathcal{F}L^{q'}$ into $M^{p_0,\infty}$. 


\section*{Acknowledgment}
K.~A.~Okoudjou  was partially supported by a RASA from the Graduate School of
UMCP,  the Alexander von Humboldt foundation, and  by a grant from the Simons Foundation ($\# 319197$ to Kasso Okoudjou). G.~E.~Pfander appreciates the hospitality of the mathematics departments at MIT and at the TU Munich. This project originated during a sabbatical at MIT and was completed during a visit of TU Munich as John von Neumann Visiting Professor.   G.~E.~Pfander also appreciates funding from the German Science Foundation (DFG) within the project Sampling of Operators.

\vspace{2cm}

\end{document}